\documentclass[final]{amsart}

\usepackage[utf8]{inputenc}
\usepackage[T1]{fontenc}
\usepackage{fixme}
\usepackage{slashed}

\theoremstyle{definition}

\newtheorem{definition}{Definition}%[section]
\newtheorem{remark}{Remark}%[section]
\newtheorem{example}{Example}%[section]

\newtheorem{proposition}{Proposition}%[section]
\newtheorem{lemma}{Lemma}%[section]
\newtheorem{theorem}{Theorem}%[section]

\newcommand{\eproof}{\begin{flushright} $\square$ \end{flushright}}

\newcommand{\im}{\mathop{\fam0 Im}\nolimits}

\newcommand{\re}{\mathop{\fam0 Re}\nolimits}

\newcommand{\bC}{{\mathbb C}}

\newcommand{\Z}{{\mathbb Z}}

\newcommand{\ra}{\mathop{\fam0 \rightarrow}\nolimits}

\newcommand{\cH}{ {\mathcal H}}

\renewcommand{\epsilon}{\varepsilon}

\renewcommand{\phi}{\varphi}

\newcommand{\abs}[1]{\lvert #1 \rvert}

\newcommand{\setZ}{\mathbb{Z}}
\newcommand{\setR}{\mathbb{R}} % added by RV
\newcommand{\setC}{\mathbb{C}}

\renewcommand{\Im}{\mathrm{Im}}
\renewcommand{\Re}{\mathrm{Re}}

\DeclareMathOperator{\GL}{GL}

\def\XXint#1#2#3{{\setbox0=\hbox{$#1{#2#3}{\int}$}
\vcenter{\hbox{$#2#3$}}\kern-.5\wd0}}

\newcommand{\Span}{\mathrm{span}}
\newcommand{\COMPLEXS}{\mathbb{C}}
\newcommand{\imun}{\mathsf{i}}
\newcommand{\INTEGERS}{\mathbb{Z}}
\newcommand{\la}{\theta}
\newcommand{\cla}{c_\la}

\newcommand{\REALS}{\mathbb{R}}

\newcommand{\TORUS}{\mathbb{T}}

\newcommand{\gfad}{\operatorname{D}}
\newcommand{\fad}{\operatorname{\Phi}}
\newcommand{\myN}{N}
\newcommand{\lca}{\mathbb{A}_\myN}
\newcommand{\bkt}[1]{\langle #1\rangle}

\begin{document}

\title{Complex Quantum Chern--Simons}
\author{J{\o}rgen Ellegaard Andersen}
\address{Center for Quantum Geometry of Moduli Spaces\\
        University of Aarhus\\
        DK-8000, Denmark}
\email{andersen@qgm.au.dk}

\author{Rinat Kashaev}
\address{University of Geneva\\
2-4 rue du Li\`evre, Case postale 64\\
 1211 Gen\`eve 4, Switzerland}
\email{rinat.kashaev@unige.ch}

\thanks{Supported in part by the center of excellence grant ``Center for quantum geometry of Moduli Spaces" from the Danish National Research Foundation, and Swiss National Science Foundation}

\begin{abstract}
We lay down a general framework for how to construct a Topological Quantum Field Theory $Z_A$ defined on shaped triangulations of orientable 3-manifolds from any Pontryagin self-dual locally compact abelian group $A$. The partition function for a triangulated manifold is given by a state integral over the LCA $A$ of a certain combinations of functions which satisfy Faddeev's operator five term relation. In the cases where  all elements of the LCA $A$ are divisible by 2 and it has a subgroup $B$ whose Pontryagin dual is isomorphic
to $A/B$, this TQFT has an alternative formulation in terms of the space of sections of a line bundle over $(A/B)^{2}$. We apply this to the LCA $\mathbb{R}\times \mathbb{Z}/\myN\mathbb{Z}$ and obtain a TQFT, which we show is  Quantum Chern--Simons theory at level $\myN$ for the complex gauge group $SL(2,\bC)$ by the use of geometric quantization. 
\end{abstract}

%\date{May 18, 2013}
\maketitle

\section{Introduction}

In this paper we develop a general method of constructing Topological Quantum Field Theories defined on shaped triangulations of orientable 3-manifolds from a Pontryagin self-dual locally compact abelian group $A$. The method is based on the notion of a quantum dilogarithm over a self-dual LCA group $A$ which generalizes and formalizes Faddeev's quantum dilogarithm which in our framework becomes the simplest example of a quantum dilogarithm over the self-dual LCA group $\REALS$.
 In the cases where  all elements of a self-dual LCA group $A$ are divisible by 2 and it has a subgroup $B$ whose Pontryagin dual is isomorphic
to $A/B$, the associated TQFT has an alternative formulation in terms of the space of sections of a line bundle over $(A/B)^{2}$. We apply this to the LCA $\mathbb{R}\times \mathbb{Z}/\myN\mathbb{Z}$ and obtain a TQFT, which we show is  quantum Chern--Simons theory at level $\myN$ for the complex gauge group $SL(2,\bC)$ by the use of geometric quantization. A similar result has been obtained by Tudor Dimofte \cite{Dimofte3d3d}. The paper is organized as follows. In Section~\ref{MS}, following \cite{MR1607296}, we recall and present the complexified ratio coordinates in Teichm\"uller space and their relationship to complexified Penner's $\lambda$-coordinates. We also recall the definition of the groupoid of decorated ideal triangulations, a convenient algebraic formalization which allows to avoid using any particularities of concrete surfaces. In Section~\ref{GQ}, we develop geometric quantization of the complex symplectic space associated to one triangle in ratio coordinates. We remark that in the paper \cite{AG2} the first author together with Gammelgaard has constructed the representation of the mapping class group also via geometric quantisation, but of higher genus moduli spaces. In the paper \cite{A}, the first author has computed the genus one representation of the mapping class group. In Section~\ref{PtRep}, the quantum representation of the groupoid of decorated ideal triangulations is discussed on the basis of the solution of the Pentagon equation. In Section~\ref{CPT}, the charged operators to be used in 3d partition functions are discussed. In Section~\ref{GF}, we present a formalism which allows to generalize our approach to potentially large class of models based on self-dual LCA groups. In Section~\ref{FoFT}, we develop a general technique which allows as to formulate the models in terms of sections of line bundles over compact LCA groups. Finally, in Section~\ref{QDN} we describe in detail the properties of the quantum dilogarithm over the the self-dual LCA group $\mathbb{R}\times \mathbb{Z}/\myN\mathbb{Z}$ underlying the quantum complex Chern--Simons with gauge group $\mathrm{GL}(2,\COMPLEXS)$.

\subsection*{Acknowledgements} W would like to thank Tudor Dimofte for interesting discussions and  for explaining his results which are similar to ours.

\section{Ratio coordinates on the $\GL(2,\setC)$ moduli space}

\label{MS}

In this section we briefly recall the constructions from \cite{MR1607296} in order to introduce the complex ratio coordinates on the moduli space of flat $\GL(2,\setC)$ connections on an oriented surface $S$ of genus $g$ with $s$ punctures, where $(2g-2+s)s>0$.

Let $\Delta(S)$ be the set of all ideal triangulations of $S$ (the isotopy classes thereof) seen as a cellular complexes and we let $\Delta_i(\tau)$ to denote the set of $i$-dimensional cells of an ideal triangulation $\tau$. Complexified \emph{Penner's decorated Teichm\"uller space} 
$\tilde{\mathcal{T}}_{\COMPLEXS}(S)$ is the set of pairs $(\tau, \lambda)$ where $\tau\in\Delta(S)$ and $\lambda\in \COMPLEXS_{\ne0}^{\Delta_1(\tau)}$, modulo the equivalence relation generated by Ptolemy relations, namely the relations $(\tau,\lambda)\sim(\tau',\lambda')$ where
$\tau$ and $\tau'$ differ by an elementary diagonal flip of a quadrilateral composed of two distinct triangles, and the values of $\lambda$ and $\lambda'$ are the same except the flipped diagonals $d\in\Delta_1(\tau)$ and $d'\in\Delta_1(\tau')$, while the values on those diagonals are related by the Ptolemy formula
\begin{equation}
\lambda(d)\lambda'(d')=ac+bd,
\end{equation}
where $a,b,c,d$ are the values of $\lambda$ (or $\lambda'$) on consecutive sides of the quadrilateral.
The moduli space of flat $PSL(2,\setC)$ connections on $S$ can be described by action of a gauge group on $\tilde{\mathcal{T}}_{\COMPLEXS}(S)$. Namely, the gauge group is $\COMPLEXS_{\ne0}^{\Delta_0(\tau)}$, and an element $f$ of it acts as follows:
\begin{equation}
(f,(\tau,\lambda))\mapsto (\tau,\lambda'), \quad \lambda'(e)=\lambda(e) f(v_1)f(v_2)
\end{equation}
where $v_1,v_2$ two punctures (possibly coinciding) connected by $e$. The remarkable fact about Penner's space and the $\lambda$ coordinates is that the pullback under the projection map of the Goldman's complex symplectic structure in the moduli space is given by a very simple formula
\begin{equation}\label{penner-form}
\omega_{\tilde{\mathcal{T}}}=\sum_{t\in\Delta_2(\tau)} \frac{da\wedge db}{ab}+\frac{db\wedge dc}{bc}+\frac{dc\wedge da}{ca}
\end{equation}
where $a,b,c$ are $\lambda$-coordinates associated with the three sides of $t$ taken in the cyclic order induced from the orientation of $S$.

The \emph{ratio} coordinates are introduced by first specifying a distinguished corner in each triangle, and then by taking two ratios of three $\lambda$-coordinates on the sides of the triangle canonically specified by the distinguished corner. Namely one writes $(b/c,a/c)$ where $c$ is the side opposite to the distinguished corner, and the cyclic order $(a,b,c)$ is induced from the orientation of $S$. We let $\tilde\Delta(S)$ to denote the set of all ideal triangulations with distinguished corners in all triangles. One defines the \emph{complexified ratio} space  $\mathcal{R}_{\COMPLEXS}(S)$ as the set of pairs $(\tau, \mu)$ where  $\tau\in\tilde\Delta(S)$ and $\mu\in \COMPLEXS_{\ne0}^{2\Delta_2(\tau)}$, modulo the equivalence relation generated by the relations $(\tau,\mu)\sim(\tau',\mu')$ where
$\tau$ and $\tau'$ either differ as before by an elementary diagonal flip of a quadrilateral composed of two distinct triangles with specific arrangement of distinguished corners, see Figure~\ref{figure1}, and the values of $\mu$ and $\mu'$ are the same except the triangles involved in the flip, and the relation between those triangles is given by the formula 
 \begin{equation}
x'=x\cdot y\equiv(x_1y_2,x_1y_2+x_2), \quad y'=x*y\equiv\left(\frac{y_1x_2}{x_1y_2+x_2},\frac{y_2}{x_1y_2+x_2}\right)
\end{equation}
where the triangles are indicated in Figure~\ref{figure1},
\begin{figure}[hBT]
  \centering
\begin{picture}(200,40)
\put(0,0){
\begin{picture}(40,40)
\put(20,0){\line(-1,1){20}}
\put(40,20){\line(-1,-1){20}}
\put(0,20){\line(1,1){20}}
\put(40,20){\line(-1,1){20}}
\put(20,0){\line(0,1){40}}
\put(20,0){\circle*{3}}
\put(0,20){\circle*{3}}
\put(20,40){\circle*{3}}
\put(40,20){\circle*{3}}
\footnotesize
\put(10,18){$x$}\put(26,18){$y$}
\put(1,18){$*$}
\put(19.5,2){$*$}
\end{picture}}
\put(160,0){\begin{picture}(40,40)
\put(20,0){\line(-1,1){20}}
\put(40,20){\line(-1,-1){20}}
\put(0,20){\line(1,1){20}}
\put(40,20){\line(-1,1){20}}
\put(0,20){\line(1,0){40}}
\put(20,0){\circle*{3}}
\put(0,20){\circle*{3}}
\put(20,40){\circle*{3}}
\put(40,20){\circle*{3}}
\footnotesize
\put(18,26){$x'$}\put(18,10){$y'$}
\put(3,20){$*$}
\put(17.5,1){$*$}
\end{picture}}
\put(95,17){$\longrightarrow$}
\end{picture}
\caption{The diagonal flip transformation for ratio coordiantes}\label{figure1}
\end{figure}
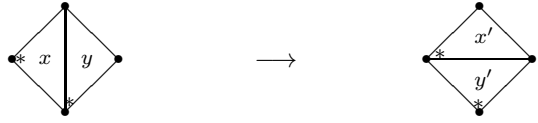
or else $\tau$ and $\tau'$ are related by a change of the distinguished corner in one particular triangle, see Figure~\ref{figure2}, and in that case $\mu$ and $\mu'$ are identical except that triangle where the relation is given by the formula
\begin{equation}
x'=\hat x\equiv (x_2/x_1,1/x_1).
\end{equation}
\begin{figure}[htb]
  \centering
\begin{picture}(200,20)
\put(0,0){\begin{picture}(40,20)
\put(0,0){\line(1,0){40}}
\put(0,0){\line(1,1){20}}
\put(20,20){\line(1,-1){20}}
\put(0,0){\circle*{3}}
\put(20,20){\circle*{3}}
\put(40,0){\circle*{3}}
\footnotesize
\put(33,0){$*$}
\put(18,5){$x$}
\end{picture}}
\put(160,0){\begin{picture}(40,20)
\put(0,0){\line(1,0){40}}
\put(0,0){\line(1,1){20}}
\put(20,20){\line(1,-1){20}}
\put(0,0){\circle*{3}}
\put(20,20){\circle*{3}}
\put(40,0){\circle*{3}}
\footnotesize
\put(17.5,14){$*$}
\put(18,5){$x'$}
\end{picture}}
\put(95,8){$\longrightarrow$}
\end{picture}
\caption{Distinguished corner change transformation}\label{figure2}
\end{figure}
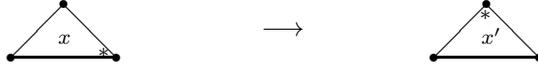
One result of \cite{MR1607296} implies that the ratio space is complex symplectic with the canonical two-form
\begin{equation}
\omega_{\mathcal{R}}=\sum_{x\in\Delta_2(\tau)} \frac{dx_1\wedge dx_2}{x_1x_2},
\end{equation}
and that its pullback by the ratio map is Penner's two-form~\eqref{penner-form},
\begin{equation}
\omega_{\tilde{\mathcal{T}}}=\varrho^*\omega_{\mathcal{R}}.
\end{equation}
Another result of \cite{MR1607296} implies that one has the following exact sequence of vector spaces
\begin{equation}
1\to \COMPLEXS_{\ne0}\stackrel{\Delta}{\to} \tilde{\mathcal{T}}_{\COMPLEXS}(S)\stackrel{\varrho}{\to} \mathcal{R}_{\COMPLEXS}(S)\stackrel{\gamma}{\to} H^1(S, \COMPLEXS_{\ne0})\to 1
\end{equation}
where $\Delta$ is the diagonal map, $\varrho$ is the ratio map described above, and $\gamma$ is a simple map constructed combinatorially, see \cite{MR1607296} for more details.

One can formalize algebraically all these constructions into a groupoid of decorated ideal triangulations, see for example, \cite{Penner2012,MR2952777}, as follows.

\begin{definition}
A \emph{decorated ideal triangulation} of $S$ is an ideal triangulation
$\tau$, where all triangles are provided with a marked corner, and a bijective ordering map
    \[
\bar\tau\colon
\{1,\ldots,2(2g-2+s)\}\ni j\mapsto\bar\tau_j\in\Delta_2(\tau)
    \]
is fixed.
\end{definition}
Graphically,  the marked corner of a triangle $\bar\tau_i$ is indicated by an asterisk and the index $i$ is put inside the triangle. The set of all decorated ideal triangulations of  $S$ is denoted by $\tilde\Delta'(S)$.

Recall that if a group $G$ freely acts on a set $X$ then there is an associated groupoid defined as follows. The objects are the $G$-orbits in $X$, while morphisms are $G$-orbits in $X\times X$ with respect to the diagonal action. Denote by $[x]$ the object represented by the element $x\in X$ and $[x,y]$ the morphism represented by the pair of elements $(x,y)\in X\times X$. Two morphisms $[x,y]$ and $[u,v]$, are composable if and only if $[y]=[u]$ and their composition is  $[x,y][u,v]=[x,gv]$, where $g\in G$ is the unique element sending $u$ to $y$. The inverse and the identity morphisms are given respectively by $[x,y]^{-1}=[y,x]$ and $\mathrm{id}_{[x]}=[x,x]$. In what follows,  products of the form $[x_1,x_2][x_2,x_3]\cdots[x_{n-1},x_n]$ will be written as $[x_1,x_2,x_3,\ldots,x_{n-1},x_n]$.

Remarking that the mapping class group $\Gamma(S)$ of $S$ freely acts on $\tilde\Delta'(S)$, denote by $\mathcal{G}(S)$ the corresponding groupoid, called the \emph{groupoid of decorated ideal triangulations}. It admits a presentation with three types of generators and four types of relations.

The generators are of the form $[\tau,\tau^\sigma]$, $[\tau,\rho_i\tau]$, and $[\tau,\omega_{ij}\tau]$, where $\tau^\sigma$ is obtained from $\tau$ by replacing the ordering map $\bar\tau$ by the map $\bar{\tau}\circ\sigma$, where $\sigma\in\mathbb{S}_{2(2g-2+s)}$ is a permutation of the set $\{1,\ldots,2(2g-2+s)\}$, $\rho_i\tau$ is obtained from $\tau$ by changing the marked corner of the triangle $\bar\tau_i$ as in Figure~\ref{figure3}, and $\omega_{ij}\tau$ is obtained from $\tau$ by applying the flip transformation in the quadrilateral composed of the triangles $\bar\tau_i$ and $\bar\tau_j$ as in Figure~\ref{figure4}.
\begin{figure}[htb]
  \centering
\begin{picture}(200,20)
\put(0,0){\begin{picture}(40,20)
\put(0,0){\line(1,0){40}}
\put(0,0){\line(1,1){20}}
\put(20,20){\line(1,-1){20}}
\put(0,0){\circle*{3}}
\put(20,20){\circle*{3}}
\put(40,0){\circle*{3}}
\footnotesize
\put(33,0){$*$}
\put(18,5){$i$}
\end{picture}}
\put(160,0){\begin{picture}(40,20)
\put(0,0){\line(1,0){40}}
\put(0,0){\line(1,1){20}}
\put(20,20){\line(1,-1){20}}
\put(0,0){\circle*{3}}
\put(20,20){\circle*{3}}
\put(40,0){\circle*{3}}
\footnotesize
\put(17.5,14){$*$}
\put(18,5){$i$}
\end{picture}}
\put(95,8){$\stackrel{\rho_i}{\longrightarrow}$}
\end{picture}
\caption{Transformation $\rho_i$.}\label{figure3}
\end{figure}
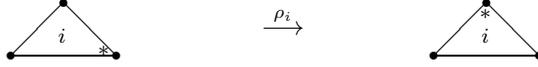
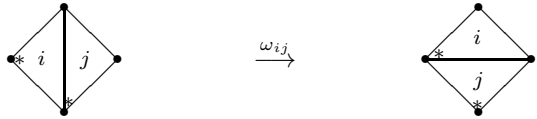
\begin{figure}[htb]
  \centering
\begin{picture}(200,40)
\put(0,0){
\begin{picture}(40,40)
\put(20,0){\line(-1,1){20}}
\put(40,20){\line(-1,-1){20}}
\put(0,20){\line(1,1){20}}
\put(40,20){\line(-1,1){20}}
\put(20,0){\line(0,1){40}}
\put(20,0){\circle*{3}}
\put(0,20){\circle*{3}}
\put(20,40){\circle*{3}}
\put(40,20){\circle*{3}}
\footnotesize
\put(10,18){$i$}\put(26,18){$j$}
\put(1,18){$*$}
\put(19.5,2){$*$}
\end{picture}}
\put(160,0){\begin{picture}(40,40)
\put(20,0){\line(-1,1){20}}
\put(40,20){\line(-1,-1){20}}
\put(0,20){\line(1,1){20}}
\put(40,20){\line(-1,1){20}}
\put(0,20){\line(1,0){40}}
\put(20,0){\circle*{3}}
\put(0,20){\circle*{3}}
\put(20,40){\circle*{3}}
\put(40,20){\circle*{3}}
\footnotesize
\put(18,26){$i$}\put(18,10){$j$}
\put(3,20){$*$}
\put(17.5,1){$*$}
\end{picture}}
\put(95,17){$\stackrel{\omega_{ij}}{\longrightarrow}$}
\end{picture}
\caption{Transformation $\omega_{ij}$.}\label{figure4}
\end{figure}

There are two sets of relations satisfied by these generators. The first set is as follows:
\begin{gather}
\label{eq:23}
[\tau,\tau^\alpha,(\tau^\alpha)^\beta]=[\tau,\tau^{\alpha\beta}],\quad \alpha,\beta\in\mathbb{S}_{2(2g-2+s)},\\
  \label{eq:cubic}
  [\tau,\rho_i\tau,\rho_i\rho_i\tau,\rho_i\rho_i\rho_i\tau]=\mathrm{id}_{[\tau]},\\\label{eq:pent}
  [\tau,\omega_{ij}\tau,\omega_{ik}\omega_{ij}\tau,\omega_{jk}\omega_{ik}\omega_{ij}\tau]
  =[\tau,\omega_{jk}\tau,\omega_{ij}\omega_{jk}\tau],
\\
\label{eq:20}
[\tau,\omega_{ij}\tau,\rho_i\omega_{ij}\tau,\omega_{ji}\rho_i\omega_{ij}\tau]
=[\tau,\tau^{(ij)},\rho_j\tau^{(ij)},\rho_i\rho_j\tau^{(ij)}].
\end{gather}
The first two relations are evident, while the other two are shown graphically in Figures~\ref{fig:pen-om},~\ref{fig:inv-rel}.
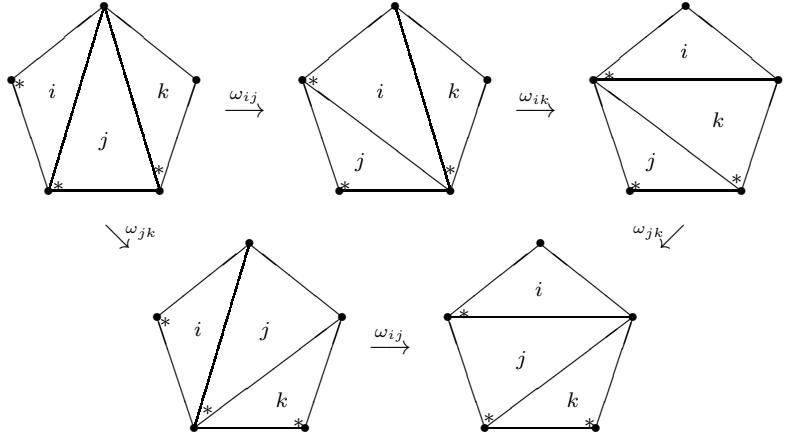
\begin{figure}[htb]
  \centering
\begin{picture}(290,160)
%pentagon 1
\put(0,90){\begin{picture}(70,70)
%sides
\put(14,0){\line(-1,3){14}}
\put(56,0){\line(1,3){14}}
\put(0,42){\line(5,4){35}}
\put(70,42){\line(-5,4){35}}
\put(14,0){\line(1,0){42}}
%diagonals
\qbezier(14,0)(20,20)(35,70)
\qbezier(56,0)(50,20)(35,70)
%vertices
\put(14,0){\circle*{3}}
\put(56,0){\circle*{3}}
\put(0,42){\circle*{3}}
\put(70,42){\circle*{3}}
\put(35,70){\circle*{3}}
%labels
\footnotesize
\put(1,38.5){$*$}
\put(15.5,0){$*$}
\put(53.75,6){$*$}
\put(14,35){$i$}
\put(33,17){$j$}
\put(55,35){$k$}
\end{picture}}
%pentagon 2
\put(110,90){\begin{picture}(70,70)
%sides
\put(14,0){\line(-1,3){14}}
\put(56,0){\line(1,3){14}}
\put(0,42){\line(5,4){35}}
\put(70,42){\line(-5,4){35}}
\put(14,0){\line(1,0){42}}
%diagonals
\put(56,0){\line(-4,3){56}}
\qbezier(56,0)(50,20)(35,70)
%vertices
\put(14,0){\circle*{3}}
\put(56,0){\circle*{3}}
\put(0,42){\circle*{3}}
\put(70,42){\circle*{3}}
\put(35,70){\circle*{3}}
%labels
\footnotesize
\put(2,39.5){$*$}
\put(14,0){$*$}
\put(53.75,6){$*$}
\put(28,35){$i$}
\put(20,9){$j$}
\put(55,35){$k$}
\end{picture}}
%pentagon 3
\put(220,90){\begin{picture}(70,70)
%sides
\put(14,0){\line(-1,3){14}}
\put(56,0){\line(1,3){14}}
\put(0,42){\line(5,4){35}}
\put(70,42){\line(-5,4){35}}
\put(14,0){\line(1,0){42}}
%diagonals
\put(56,0){\line(-4,3){56}}
\put(0,42){\line(1,0){70}}
%vertices
\put(14,0){\circle*{3}}
\put(56,0){\circle*{3}}
\put(0,42){\circle*{3}}
\put(70,42){\circle*{3}}
\put(35,70){\circle*{3}}
%labels
\footnotesize
\put(4,41.5){$*$}
\put(14,0){$*$}
\put(52.25,2.5){$*$}
\put(33,50){$i$}
\put(20,9){$j$}
\put(45,24 ){$k$}
\end{picture}}
%pentagon4
\put(55,0){\begin{picture}(70,70)
%sides
\put(14,0){\line(-1,3){14}}
\put(56,0){\line(1,3){14}}
\put(0,42){\line(5,4){35}}
\put(70,42){\line(-5,4){35}}
\put(14,0){\line(1,0){42}}
%diagonals
\qbezier(14,0)(20,20)(35,70)
\put(14,0){\line(4,3){56}}
%vertices
\put(14,0){\circle*{3}}
\put(56,0){\circle*{3}}
\put(0,42){\circle*{3}}
\put(70,42){\circle*{3}}
\put(35,70){\circle*{3}}
%labels
\footnotesize
\put(1,38.5){$*$}
\put(17,5){$*$}
\put(51.5,0){$*$}
\put(14,35){$i$}
\put(39,35){$j$}
\put(45,8){$k$}
\end{picture}}
%pentagon5
\put(165,0){\begin{picture}(70,70)
%sides
\put(14,0){\line(-1,3){14}}
\put(56,0){\line(1,3){14}}
\put(0,42){\line(5,4){35}}
\put(70,42){\line(-5,4){35}}
\put(14,0){\line(1,0){42}}
%diagonals
\put(0,42){\line(1,0){70}}
\put(14,0){\line(4,3){56}}
%vertices
\put(14,0){\circle*{3}}
\put(56,0){\circle*{3}}
\put(0,42){\circle*{3}}
\put(70,42){\circle*{3}}
\put(35,70){\circle*{3}}
%labels
\footnotesize
\put(4,41.5){$*$}
\put(13.5,2){$*$}
\put(51.5,0){$*$}
\put(33,50){$i$}
\put(26,24){$j$}
\put(45,8){$k$}
\end{picture}}
%arrows with labels
\put(35,70){$\searrow$}\put(43,74){\tiny$\omega_{jk}$}
\put(245,70){$\swarrow$}\put(235,74){\tiny$\omega_{jk}$}
\put(135,28){$\stackrel{\omega_{ij}}{\longrightarrow}$}
\put(80,118){$\stackrel{\omega_{ij}}{\longrightarrow}$}
\put(190,118){$\stackrel{\omega_{ik}}{\longrightarrow}$}
\end{picture}
  \caption{Pentagon relation~\eqref{eq:pent}.}
  \label{fig:pen-om}
\end{figure}

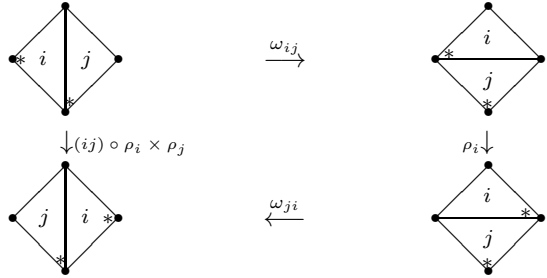
\begin{figure}[htb]
  \centering
  \begin{picture}(200,100)
%upper row
\put(0,60){\begin{picture}(200,40)
\put(0,0){\begin{picture}(40,40)
\put(20,0){\line(-1,1){20}}
\put(40,20){\line(-1,-1){20}}
\put(0,20){\line(1,1){20}}
\put(40,20){\line(-1,1){20}}
\put(20,0){\line(0,1){40}}
\put(20,0){\circle*{3}}
\put(0,20){\circle*{3}}
\put(20,40){\circle*{3}}
\put(40,20){\circle*{3}}
\footnotesize
\put(10,18){$i$}\put(26,18){$j$}
\put(1,18){$*$}
\put(19.5,2){$*$}
\end{picture}}
\put(160,0){\begin{picture}(40,40)
\put(20,0){\line(-1,1){20}}
\put(40,20){\line(-1,-1){20}}
\put(0,20){\line(1,1){20}}
\put(40,20){\line(-1,1){20}}
\put(0,20){\line(1,0){40}}
\put(20,0){\circle*{3}}
\put(0,20){\circle*{3}}
\put(20,40){\circle*{3}}
\put(40,20){\circle*{3}}
\footnotesize
\put(18,26){$i$}\put(18,10){$j$}
\put(3,20){$*$}
\put(17.5,1){$*$}
\end{picture}}
\put(95,17){$\stackrel{\omega_{ij}}{\longrightarrow}$}
\end{picture}}
\put(18,46){$\downarrow$\tiny$(ij)\circ\rho_i\times\rho_j$}
\put(170,46){{\tiny$\rho_i$}$\downarrow$}
%lower row
\put(0,0){\begin{picture}(200,40)
\put(0,0){\begin{picture}(40,40)
\put(20,0){\line(-1,1){20}}
\put(40,20){\line(-1,-1){20}}
\put(0,20){\line(1,1){20}}
\put(40,20){\line(-1,1){20}}
\put(20,0){\line(0,1){40}}
\put(20,0){\circle*{3}}
\put(0,20){\circle*{3}}
\put(20,40){\circle*{3}}
\put(40,20){\circle*{3}}
\footnotesize
\put(10,18){$j$}\put(26,18){$i$}
\put(16,2){$*$}
\put(34,17.5){$*$}
\end{picture}}
\put(160,0){\begin{picture}(40,40)
\put(20,0){\line(-1,1){20}}
\put(40,20){\line(-1,-1){20}}
\put(0,20){\line(1,1){20}}
\put(40,20){\line(-1,1){20}}
\put(0,20){\line(1,0){40}}
\put(20,0){\circle*{3}}
\put(0,20){\circle*{3}}
\put(20,40){\circle*{3}}
\put(40,20){\circle*{3}}
\footnotesize
\put(18,26){$i$}\put(18,10){$j$}
\put(32,20){$*$}
\put(17.5,1){$*$}
\end{picture}}
\put(95,17){$\stackrel{\omega_{ji}}{\longleftarrow}$}
\end{picture}}
\end{picture}
  \caption{Inversion relation~\eqref{eq:20}.}
  \label{fig:inv-rel}
\end{figure}
The following commutation relations fulfill the remaining second set of relations:
\begin{gather}
[\tau,\rho_i\tau,(\rho_i\tau)^\sigma]=[\tau,\tau^\sigma,\rho_{\sigma^{-1}(i)}\tau^\sigma],\\
[\tau,\omega_{ij}\tau,(\omega_{ij}\tau)^\sigma]=[\tau,\tau^\sigma,
\omega_{\sigma^{-1}(i)\sigma^{-1}(i)}\tau^\sigma],\\
[\tau,\rho_j\tau,\rho_i\rho_j\tau]=[\tau,\rho_i\tau,\rho_j\rho_i\tau],\\
[\tau,\rho_i\tau,\omega_{jk}\rho_i\tau]=[\tau,\omega_{jk}\tau,\rho_i\omega_{jk}\tau],\ i\not\in\{j,k\},\\
[\tau,\omega_{ij}\tau,\omega_{kl}\omega_{ij}\tau]=[\tau,\omega_{kl}\tau,\omega_{ij}\omega_{kl}\tau],\
\{i,j\}\cap\{k,l\}=\emptyset.
\end{gather}

\section{The complex Weil representation via geometric quantisation}

\label{GQ}

Motivated by the previous section, we will now consider the geometric quantisation of the space we associated to a triangle above, e.g. $(\setC^*)^2$ with the symplectic form $\tilde \Omega = \frac{dx}{x}\wedge\frac{dy}{y}$, where $(x,y)$ are coordinates on $(\setC^*)^2$. Consider the corresponding logarithmic coordinates $(u,v)$ on $\setC^2$ with the underlying real coordinates 
$$  u = u' + iu''\quad, \quad v = v' + iv''$$
such that we have the covering map
$$\pi : \setC^2 \ra (\setC^*)^2$$
given by
$$ \pi(u,v) = (\exp(2\pi i u), \exp(2 \pi i v)) = (x,y).$$
So we see that
$$ \pi^*\tilde \Omega = \Omega$$
where $\Omega = du\wedge dv.$
For any non-zero complex number $t$ consider the following real symplectic structure
\begin{align}
  \nonumber
  \Omega_t &= \frac{t}{2} \Omega + \frac{\bar t}{2} \bar \Omega.
\end{align}
In real coordinates we have that
\begin{align}
  \nonumber
    \Omega_t &= \re(t) (du'\wedge dv' - du'' \wedge dv'') - \im(t) (du''
  \wedge dv' + du' \wedge dv'')
\end{align}
Hence we get the following formulae for the Hamiltonian vector fields of the real coordinates.
\begin{align}
  \nonumber
    X_{u'} &= \quad\frac{1}{\abs{t}^2} (\re(t) \frac{\partial}{\partial v'} -
  \im(t)\frac{\partial}{\partial v''})\\
  \nonumber
    X_{u''} &= -\frac{1}{\abs{t}^2} (\im(t) \frac{\partial}{\partial v'} +
  \re(t)\frac{\partial}{\partial v''})\\
  \nonumber
    X_{v'} &= -\frac{1}{\abs{t}^2} (\re(t) \frac{\partial}{\partial u'} -
  \im(t)\frac{\partial}{\partial u''})\\
  \nonumber
    X_{v''} &= \quad\frac{1}{\abs{t}^2} (\im(t) \frac{\partial}{\partial u'} +
  \re(t)\frac{\partial}{\partial u''})
\end{align}
We shall now fix a primitive $\alpha_t$ of $\Omega_t$ given by
$$\alpha_t =  \frac{t}{2} \alpha + \frac{\bar t}{2} \bar \alpha$$
where $\alpha = \frac12(v du - u dv)$.
Consider the trivial line bundle over $\tilde L$ over $\setC^2$.
Let $\tilde \nabla$ be the connection associated  to the primitive $\alpha_t$ on $\tilde L$,
i.e
$$\tilde \nabla = \nabla^t + 2 \pi i \alpha_t.$$
 such that 
 $$F_{\nabla} = 2\pi i \Omega_t.$$
Now we consider the canonical pre-quantum operators
$$P_f = -i\tilde\nabla_{X_f}+f$$
which act on $\widetilde\cH =
C^\infty(\setC^2,\tilde L)$ and satisfy the commutation relation
\begin{align}\label{pqocr}
  %\nonumber
    [P_f, P_g] &= -iP_{\{f, g\}}
\end{align}
In order to complete the geometric quantisation program, we need to specify a polarization ${\mathcal P}$, given by the following real Lagrangian subspace of $\setC^2$.
$${\mathcal P}= \Span_{\mathbb R}\{\frac{\partial}{\partial u}, \frac{\partial}{\partial \bar v}\}.$$
Hence the space of polarised sections is given by
$$
  \widetilde\cH_{{\mathcal P}} = \{ s \in \widetilde\cH \mid  \tilde\nabla_{\frac{\partial}{\partial u}} s = 
  \tilde\nabla_{\frac{\partial}{\partial \bar v}}s = 0 \}.
$$
We observe that polarised sections are determined by their restrictions to the transversal
\begin{align}
  \nonumber
    \setR^2 = \{ (u,v)\in \setC^2 \mid u'' = v'' = 0 \}.
\end{align}
With this transversal, we can now describe the pre-quantum operators of the basic coordinate functions
\begin{align}
  %\nonumber
  \label{pro1}
    P_{u'}  &= -\frac{i}{t}\tilde\nabla_{\frac{\partial}{\partial v'}} + u' &
  P_{u''} &= -\frac{1}{t}\tilde\nabla_{\frac{\partial}{\partial v'}} \\
 % \nonumber
  P_{v'}  &= -\frac{i}{\bar t}\tilde\nabla_{\frac{\partial}{\partial u'}} + v' &
  P_{v''} &= -\frac{1}{\bar t}\tilde\nabla_{\frac{\partial}{\partial u'}} \label{pro2}
\end{align}
When we restrict the symplectic form to the transversal, we get the symplectic form
\begin{align}
  \nonumber
    (\Omega_t)|_{\setR^2} = \re(t) \omega
\end{align}
where $\omega=du'\wedge dv'$.
Since we are in the process of quantising $(\setC^*)^2$, we also need to consider the action of the lattice $\setZ^2 \subset \setR^2$ on the restriction of the line bundle $\tilde L$ and its connection $\tilde\nabla$ to the transversal $\setR^2$. From this we get immediately a quantisation condition, namely that 
$$\re(t) \in 2\pi \setZ.$$
Let us denote this integer $k = \frac{\re(t)}{2\pi}$ and call it the {\em level} of the theory. We can now construct a unique line bundle $L$ and a connection $\nabla$ on $T^2 = \setR^2/\setZ^2$ induced from $\tilde L$ and $\tilde \nabla$ by specifying the lift of the action of the $\setZ^2$ to $\tilde L|_{\setR^2}$ by specifying the following multipliers
\begin{align}
  \nonumber
    e_{(1,0)}(u',v') &= e^{\pi iv'}\\
  \nonumber
  e_{(0,1)}(u',v') &= e^{-\pi iu'}
\end{align}
We observe that the first Chern class of this line bundle $c_1(L) = [\omega]\in H^2(T^2,\setZ)$
 is a generator.
We now consider the pre-Hilbert space of the theory $\cH = C^\infty(T^2, L^k)$, which is canonically given by
$$ \cH = \{ s \in C^\infty(\setR^2) | s(u'+1, v') =  e^{\pi ikv'}s(u', v'), \\ s(u', v'+1) = e^{-\pi iku'}s(u', v')\}.$$
The Hilbert space structure we will consider on this space is the following one
\begin{align}\label{US}
(s_1,s_2) = \int_{T^2} s_1\overline{S(s_2)}\omega,
\end{align}
where $S : T^2 \ra T^2$ is the diffeomorphism given by 
$$S(u',v') = (-v',u').$$
We observe that $S : \cH \ra \cH$, thus $(\cdot,\cdot)$ is the usual $L_2$-inner product on $C^\infty(T^2,L^k)$ twisted by symplectomorphism $S$ of $T^2$.
Let $\nabla^{(k)}$ be the induced connection in $L^k$ with curvature $k\omega$, which is explicitly given by
\begin{align}
  \nonumber
    \nabla = \nabla^t + i k \pi(v'du' - u'dv').
\end{align}
Introduce the operators
$$  A = \nabla_{\frac{\partial}{\partial u'}} - 2\pi i kv' \text{ and }B = \nabla_{\frac{\partial}{\partial v'}} + 2\pi ik u'.$$
Then we observe that
$$  \cH = \{ s \in C^\infty(\setR^2) | e^A s = s, e^B s = s\}.$$
We do not have that the pre-quantum operators associated to the coordinate functions (\ref{pro1})-(\ref{pro2}) preserve $\cH$. However, we proceeds as follows.

Introduce a new non-zero complex variable $b$ related to $t$ by
$$ t = 4\pi  b^2$$
thus $2\re(b^2) \in \setZ$. Consider the operators
\begin{align}
  \nonumber
  q &= ib P_{u}& 
  p &= ib P_{v}\\
  \nonumber
 \tilde q &=  -i\bar b P_{\bar u} &
 \tilde p  &= i\bar b P_{\bar v}
\end{align}
and their exponentials
\begin{align}
  \nonumber
  U &= e^{2\pi b^{-1}q} & V &= e^{2\pi b^{-1}p}\\
  \nonumber
  \tilde U &= e^{2\pi \bar b^{-1}\tilde q} &
  \tilde V &= e^{2\pi \bar b^{-1} \tilde p}
\end{align}
all acting on $C^\infty(\setR^2)$.
\begin{lemma}
We have the commutation relations
\begin{align}
  \nonumber
  [q,p] &= -\frac{1}{2\pi i} = [\tilde q, \tilde p]\\
  \nonumber
  [q, \tilde q] &= [q,\tilde p] = [p, \tilde q] = 0.
\end{align}
Furthermore their exponentials $U,V,\tilde U,\tilde V$ preserves a dense subspace of $\cH$.
\end{lemma}
\proof
The commutation relations follow very simply from the general relation (\ref{pqocr}). The scalings are precisely chosen so that zero order part of $q, p, \tilde q, \tilde p$ exponentiate to well-defined functions on $T^2$. 
We now compute
$$
 e^{a\nabla_{\frac{\partial}{\partial v'}}}(s)(u',v') = e^{-\pi i k a u'} s(u', v'+a).$$
 Thus we see that this operator is only defined on the subspace of analytic sections, which can be extended so that this evaluation makes sense. For such $s$ it is easy to see that $\tilde s$ is again in $\cH$.
\eproof
We observe that
\begin{align}
  \nonumber
  A &=  -2\pi \bar b\tilde q - 2\pi b p
\end{align}
and
\begin{align}
  \nonumber
  B &=  2\pi  b q - 2\pi \bar b\tilde p.
\end{align}

Hence we have that
\begin{align}
\label{nf}
\cH = \{ s\in C^\infty(\setR^2) \mid U^{b^2} s = \tilde U^{-\bar b^2}s, V^{b^2} s = \tilde V^{\bar b^2}s\}.
\end{align}

\begin{proposition}\label{spec}
The unbounded operators $q,p,\tilde q, \tilde p$ acting on the Hilbert space
$$ \cH' = \{ s\in H_1([0,1]^2) \mid  s(1, v') =  e^{\pi ikv'}s(0, v'), \\ s(u', 1) = e^{-\pi iku'}s(u', 0)\}$$
are normal and their Domain and spectrum are as follows
  \begin{align}
  {\mathcal D}(q) &= \{s\in\cH' \mid s(0,v') = 0\} & \sigma(q) &= i ( b- b^{-1}k) [0,1] + b^ {-1} \setZ\\
   {\mathcal D}(\tilde q) &= \{s\in\cH' \mid s(u',0) = 0\} & \sigma(\tilde q) &= i ( \bar b- \bar b^{-1}k) [0,1] + \bar b^ {-1} \setZ\\
   {\mathcal D}(p) &= \{s\in\cH' \mid s(0,v') = 0\} & \sigma(q) &= i b [0,1] \\
   {\mathcal D}(\tilde p) &= \{s\in\cH' \mid s(u',0) = 0\} & \sigma(\tilde q) &= - i  \bar b [0,1].
  \end{align}
  
\end{proposition}

\begin{remark}
First of all we observe that the intersection of these two domains is dense in the Hilbert space of $L_2$ sections of $L^k$ over $T^2$. The spectral theorem for unbounded normal operators therefore provides us with a spectral calculus for these operators (see e.g. section 5.3.17 - 19 in \cite{ANOW}). We shall make use of this in the following section where we consider the Ptolemy transformation. Because of the above mentioned denseness, it is enough to establish identities for bounded operators on this intersection (or on dense subspaces of this intersection), for them to hold on the entire space of $L_2$ sections of $L^k$ over $T^2$.
\end{remark}

\proof
We will provide the proof for a slight generalisation of the operator $q$ and leave the rest to the reader. Consider the operator ($a\not=0$)
\begin{align*}
  D &= a\nabla_{\frac{\partial}{\partial v'}} + cu' + dv'
\end{align*}
on the space
\begin{align*}
  \mathcal{H}' &= \{ s\in H_1([0,1]^2) \mid s(0,v') = e^{\pi ikv'}
  s(1,v'), s(u',0) = e^{-\pi iku'}s(u',1) \}.
\end{align*}

For $\widetilde s \in\mathcal{H}'$, we seek a solution
$s\in\mathcal{H}'$ to the following equation
\begin{align}
  \label{eq:10}
  (D-\lambda)s = \widetilde s
\end{align}
for $\lambda\in\setC$.

First we examine the equation
\begin{align}
  \label{eq:1}
  (D-\lambda)s &= 0.
\end{align}
We see that any solution has to be of the form
\begin{align*}
  s(u',v') &= r(u',v') e^{i\pi ku'v' -
    (\frac{c}{a}u'+\frac{1}{2}\frac{d}{a})v' + \frac{2}{a}v'}
\end{align*}
where
\begin{align*}
  (D-\lambda)s &= \frac{\partial r}{\partial v'}(u',v') e^{i\pi
    ku'v'-\frac{c}{a}u'v'-\frac{1}{2}\frac{d}{a}v'^2+\frac{2}{a}v'}\\
  &+ a(i\pi ku'-(\frac{c}{a}u' - \frac{d}{a}v' + \frac{2}{a}))s\\
  &- (i\pi kau'-cu' - dv' + \lambda)s\\
  &= 0
\end{align*}
so we need that
\begin{align*}
  \frac{\partial r}{\partial v'}(u',v') &= 0.
\end{align*}
Hence the general solution to~\eqref{eq:1} is given by the expression
\begin{align}
  s(u',v') &= r(u') e^{i\pi ku'v' - \frac{c}{a}u'v' -
    \frac{1}{2}\frac{d}{a}v'^2 + \frac{\lambda}{a}v'}
\end{align}
where $c$ is any function such that $s \in H_1([0,1]^2)$.
Now we check that
\begin{align}
\nonumber
  s(u'+1,v') &= r(u'+1)e^{i\pi k(u'+1)v' - \frac{c}{a}(u'+1)v' 
    - \frac{1}{2}\frac{d}{a}v'^2 + \frac{2}{a}v'}\\
    \nonumber
  &= e^{\pi ikv'} s(u',v')
\end{align}
So we need
\begin{align}
\nonumber
  r(u'+1) e^{-\frac{c}{a}v'} &= r(u')
\end{align}
and
\begin{align}
\nonumber
  S(u',v'+1) &= r(u')e^{i\pi ku'(v'+1) - \frac{c}{a}u'(v'+1) +
    \frac{\lambda}{a}(v'+1) - \frac{1}{2}\frac{d}{a}(v'+1)^2}\\
    \nonumber
  &= e^{-\pi iku'} s(u',v')
\end{align}
hence we must have that
\begin{align}
 \nonumber% \label{eq:6}
  r(u') e^{2\pi iku' - \frac{c}{a}u' + \frac{\lambda}{a} -
    \frac{d}{a}v' - \frac{1}{2}\frac{d}{a}} &= r(u')
\end{align}
from which we conclude that either $c\equiv 0$ or
\begin{align}
  \nonumber%\label{eq:7}
  \frac{c}{a} &= 0 &
  \frac{d}{a} &= 0
\end{align}
and
\begin{align}
 \nonumber% \label{eq:8}
  k &= 0 \qquad\text{and}\qquad  \frac{\lambda}{a} \in 2\pi i\setZ
\end{align}
Unless $c=0,k=0$ and $\lambda\in 2\pi ia\setZ$ then~\eqref{eq:1} only
has the zero solution. Otherwise it has the solutions
\begin{align}
  \nonumber%\label{eq:9}
  s(u',v') &= r(u')e^{2\pi iuv'}
\end{align}
where $r(u'+1) = r(u')$.

Returning to~\eqref{eq:10}, any solution can be written on the form
\begin{align*}
  s(u',v') &= r(u',v')e^{i\pi ku'v' - \frac{c}{a}u'v' -
    \frac{1}{2}\frac{d}{a}v'^2 + \frac{\lambda}{a}v'}
\end{align*}
where
\begin{align*}
  \frac{\partial r}{\partial v'}(u',v') &= e^{-i\pi ku'v' +
    \frac{c}{a}u'v' + \frac{1}{2}\frac{d}{a}v'^2 -
    \frac{\lambda}{a}u'} \widetilde s(u',v')
\end{align*}

Thus
\begin{align*}
  r(u',v') &= \tilde r(u') + \int_0^{v'} e^{-i\pi ku'\widetilde v'
    + \frac{c}{a}u'\widetilde v'
    - \frac{1}{2}\frac{d}{a}\widetilde v'^2 
    - \frac{\lambda}{a}\widetilde v'} \widetilde s(u',\widetilde
  v')d\widetilde v'
\end{align*}
Now we compute
\begin{align*}
  s(u'+1,v') &= r(u'+1,v') e^{i\pi kv'} e^{-\frac{c}{a}v'} e^{i\pi
    ku'v' - \frac{c}{a}u'v' + \frac{\lambda}{a}v' -
    \frac{1}{2}\frac{d}{a}v'^2}\\
  &= e^{\pi ikv'}s(u',v')
\end{align*}
so we must have
\begin{align*}
  \tilde r(u'+1) &+ \int_0^{v'} e^{-i\pi k(u'+1)\widetilde v' +
    \frac{c}{a}(u'+1)\widetilde v' - \frac{1}{2}\frac{d}{a}\widetilde
    v'^2 - \frac{\lambda}{a}\widetilde v'} e^{ik\pi \widetilde v'}
  \widetilde s(u', \widetilde v') dv'\\
  &= e^{\frac{c}{a}v'}(\tilde r(u') + \int_0^{v'} e^{-ik\pi
    u'\widetilde v' + \frac{c}{a}u'\widetilde v' +
    \frac{1}{2}\frac{d}{a}v'^2 + \frac{\lambda}{a}\widetilde v'}
  s(u',\widetilde v')d\widetilde v')
\end{align*}
Thus
\begin{align}
  \label{eq:11}
  \tilde r(u'+1) - e^{\frac{c}{a}v'}\widetilde r(u') &= 
  -\int_0^{v'} (e^{\frac{c}{a}v'} - e^{\frac{c}{a}\widetilde v'})
  e^{-ik\pi u'\widetilde v' + \frac{1}{2}\frac{d}{a}\widetilde v'^2 +
    \frac{c}{a}u'\widetilde v' - \frac{\lambda}{a}\widetilde v'}
  s(u',\widetilde v')d\widetilde v'
\end{align}
Further we find that
\begin{align*}
  s(u',v'+1) &= r(u',v'+1)e^{i\pi ku'} e^{-\frac{c}{a}u'}
  e^{\frac{\lambda}{a}} e^{-\frac{d}{a}v'} e^{-\frac{d}{2a}} 
  e^{i\pi ku'v'-\frac{c}{a}u'v' + \frac{\lambda}{a}v' -
    \frac{1}{2}\frac{d}{a}v'^2}\\
  &= e^{-\pi iku'} s(u',v')
\end{align*}
hence we require
\begin{align}
  \nonumber
  \widetilde r(u') &+ \int_0^{v'+1} e^{-i\pi ku'\widetilde v' +
    \frac{1}{2}\frac{d}{a}\widetilde v'^2 + \frac{c}{a}u'\widetilde v'
    - \frac{\lambda}{a}\widetilde v'} \widetilde s(u',\widetilde
  v')d\widetilde v'\\
  \label{eq:12}
  &= e^{-2\pi iku'} e^{\frac{c}{a}u'} e^{-\frac{\lambda}{a}}
  e^{\frac{d}{a}v'} e^{\frac{d}{2a}}
  (\tilde r(u') + \int_0^{v'} e^{-i\pi ku'\widetilde v' +
    \frac{1}{2}\frac{d}{a}\widetilde v'^2 + \frac{c}{a}u'\widetilde v'
    - \frac{\lambda}{a}\widetilde v'} \widetilde s(u', \widetilde v')
  d\widetilde v')
\end{align}

Now we let $u' = 0$ in~\eqref{eq:11} we obtain
\begin{align*}
  \tilde r(1) - e^{\frac{c}{a}v'}\tilde r(0) &= \int_{0}^{v'}
  (e^{\frac{c}{a}v'} - e^{\frac{c}{a}\widetilde v'})
  e^{-\frac{\lambda}{a}\widetilde v' +
    \frac{1}{2}\frac{d}{a}\widetilde v'^2} s(0,\widetilde v')
  d\widetilde v'
\end{align*}
and setting $v' = 0$ in~\eqref{eq:12} we get
\begin{align*}
  (1 - e^{-2\pi iku'} e^{\frac{c}{a}u'} e^{-\frac{\lambda}{a}}
  e^{\frac{d}{2a}})\tilde r(u') &= -\int_0^1 
  e^{-i\pi ku'\widetilde v' +
    \frac{1}{2}\frac{d}{a}\widetilde v'^2 + \frac{c}{a}u'\widetilde v'
    - \frac{\lambda}{a}\widetilde v'} \widetilde s(u',\widetilde
  v')d\widetilde v'
\end{align*}
Thus if
\begin{align*}
  (c-2\pi iak)[0,1] + d[0,1] + \frac{d}{2} + 2\pi ia\setZ\\
  \lambda\not\in (c-2\pi iak)[0,1] + 2\pi ia\setZ
\end{align*}
Then we can solve for $\tilde r$
\begin{align*}
  \tilde r(u') &= \frac{1}{e^{-2\pi iku'} e^{\frac{c}{a}u'}
    e^{-\frac{\lambda}{a}} e^{\frac{d}{a}v'} e^{\frac{d}{2a}} - 1}
  \int_0^1 e^{-i\pi ku'\widetilde v' +
    \frac{1}{2}\frac{d}{a}\widetilde v'^2 + \frac{c}{a}u'\widetilde v'
  - \frac{\lambda}{a}\widetilde v'}s(u',\widetilde v') d\widetilde v'
\end{align*}
But going back to equation~\eqref{eq:11} for $u'=0$ we proceed by
differentiating in $v'$ to obtain
\begin{align*}
  -\frac{c}{a}e^{\frac{c}{a}v'} \tilde r(0) &= \frac{c}{a}
  e^{\frac{c}{a}v'} \int_{0}^{v'} e^{-\frac{\lambda}{a}\widetilde v' +
    \frac{1}{2}\frac{d}{a}\widetilde v'^2} s(0,\widetilde v')d\widetilde
  v' 
  + e^{\frac{c}{a}v'} e^{-\frac{\lambda}{a}v' +
    \frac{1}{2}\frac{d}{a}v'^2 } s(0,v')
  - e^{\frac{c}{a}v'} e^{-\frac{\lambda}{a}v' +
    \frac{1}{2}\frac{d}{a}v'^2} s(0,v')
\end{align*}
hence
\begin{align*}
  -\tilde r(0) &= \int_0^{v'} e^{-\frac{\lambda}{a}\widetilde v' +
    \frac{1}{2}\frac{d}{a}\widetilde v'^2} s(0,\widetilde v')
  d\widetilde v'
\end{align*}
which implies that we need to require that
\begin{align}
  \label{eq:14}
  s(0, \widetilde v') &= 0
\end{align}
and then $\tilde r(0) = \tilde r(1) = 0$.

Returning to our formula for $\tilde r$, we see that~\eqref{eq:14}
does indeed imply $\tilde r(0)=\tilde r(1) = 0$.
Thus we conclude that $D$ with the domain
$${\mathcal D}(D) = \{s\in\cH' \mid s(0,v') = 0\}$$
has spectrum
$$ \sigma(D) = (c-2\pi iak)[0,1] + 2\pi ia\setZ.$$

\eproof
We see immediately from the definition of our unitary structure that we have the following proposition.

\begin{proposition}
We have that the following adjoint relations with respect to the unitary structure (\ref{US})
$$ q^* = \tilde q, p^* = \tilde p.$$
\end{proposition}

We shall now make contact between the geometric constructions in this section and the LCA approach to these TQFT's we have presented in Section \ref{GF}. We do this via the $k$'th order Weil-Gel'fand-Zak transform given in the following proposition.

\begin{proposition}
  \label{prop:1}
  We have an isomorphism
  \begin{align}
  \nonumber
      W^{(k)}\colon S(\setR)\otimes \setC^k \longrightarrow
      \cH
  \end{align}
  given by
  \begin{align}
  \nonumber
      W^{(k)}(f_0, \ldots, f_{k-1})(u',v') &= e^{-k\pi iu'v'}
      \sum_{j=0}^{k-1}\sum_{m\in \setZ} f_j(u' + \frac{m}{k}) e^{-2\pi
        imv'} e^{2\pi iju'},
  \end{align}
  which extend to an isometry with respect to the following inner product on $S(\setR)\otimes \setC^k$.
\end{proposition}
\begin{proof}
  We need to check that $W^{(k)}$ maps into $\cH$. Let $(f_0,\ldots f_{k-1})\in S(\setR)\otimes \setC^k$. Then it is clear that $W^{(k)}(f_0,\ldots, f_{k-1})$ is a well-defined smooth function on $\setR^2$ and we compute
  \begin{align}
  \nonumber
    W^{(k)}(f_0, \ldots, f_{k-1})(u'+1,v') &= e^{-\pi ik (u'+1)v'}
    \sum_{j=0}^{k-1}\sum_{m\in \setZ} f_j(u' + 1 + \frac{m}{k})
    e^{-2\pi imv'} e^{2\pi ij(u'+1)} \\
  \nonumber
    &= e^{-\pi ikv'}e^{-\pi iku'v'} \sum_{j=0}^{k-1}\sum_{m'\in\setZ}
    f_j(u'+\frac{m'}{k}) e^{-2\pi i(m'-k)v'}e^{2\pi iju'}\\
  \nonumber
    &= e^{\pi ikv'} W^{(k)}(f_0,\ldots,f_{k-1})(u',v')
  \end{align}
  and
  \begin{align}
  \nonumber
    W^{(k)}(f_0, \ldots, f_{k-1})(u', v'+1) &= e^{-\pi iku'}
    W^{(k)}(f_0, \ldots, f_{k-1})(u',v').
  \end{align}
  Hence we conclude that $W^{(k)}(f_0,\ldots,f_{k-1})\in \cH$.
 
Let $f_j(u') = f_j(u')e^{2\pi iju'}$ and define
\begin{align}
  \nonumber
  g_{j'}(u') &= \int_0^1 W^{(k)}(f_0,\ldots,f_{k-1})(u',v')e^{k\pi
    iu'v'}e^{2\pi ij'v'}dv'\\
  \nonumber
  &= \sum_{j=0}^{k-1}\sum_{m\in\setZ} \widetilde
  f_j(u'+\tfrac{m}{k})e^{-2\pi ij\frac{m}{k}} \int_0^1 e^{2\pi
    i(j'-m)v'} dv'\\
  \nonumber
  &= \sum_{j=0}^{k-1} \widetilde f_j(u'+\tfrac{j'}{k})e^{-2\pi
    i\frac{jj'}{k}}, \quad j'=0,\ldots k-1
\end{align}
Now we let 
\begin{align}
  \nonumber
  \widetilde g_{j'}(u') &= g_{j'}(u' - \tfrac{j'}{k})
  = \sum_{j=0}^{k-1} \widetilde f_j(u')e^{-2\pi i\frac{jj'}{k}}
\end{align}
Then
\begin{align}
  \nonumber
  \frac{1}{k}\sum_{j'=0}^{k-1} \widetilde g_{j'}(u') e^{2\pi
    i\frac{j_0j'}{k}} &=
  \frac{1}{k}\sum_{j=0}^{k-1} \widetilde f_j(u') \sum_{j'=0}^{k-1}
  e^{2\pi i\frac{(j_0-j)j'}{k}}\\
  \nonumber
  &= \sum_{j=0}^{k-1} \widetilde f_j(u') \delta_{j,j_0} \\
   \nonumber
  &= f_{j_0}(u')e^{2\pi ij_0u'}
\end{align}
So
\begin{align}
  \nonumber
  f_{j_0}(u') =  & e^{-2\pi ij_0u'}\cdot \\
  \nonumber
&\frac{1}{k}\sum_{j'=0}^{k-1}
  e^{2\pi i\frac{j_0j'}{k}} \int_0^1
  W^{(k)}(f_0,\ldots,f_{k-1})(u'-\tfrac{j'}{k}, v') e^{k\pi
    i(u'-\frac{j'}{k})v'} e^{2\pi ij'v'}dv'
\end{align}
Now compute any for $s\in \cH$
\begin{align}
  \nonumber
  f_{j_0}(u') &= e^{-2\pi ij_0u'} \frac{1}{k}\sum_{j'=0}^{k-1} e^{2\pi
    i\frac{j_0j'}{k}} \int_0^1 s(u'-\tfrac{j'}{k}, v'') e^{k\pi
    i(u'-\frac{j'}{k})v''} e^{2\pi ijv''}dv''
\end{align}
and so we get that
\begin{align}
  \nonumber
  W^{(k)}(f_0,\ldots,f_{k-1})(u',v')
  &= e^{-k\pi
    iu'v'}\sum_{j_0=0}^{k-1}\sum_{m\in\setZ} e^{-2\pi ij_0\frac{m}{k}}
   \Biggl[\frac{1}{k} \sum_{j=0}^{k-1} e^{2\pi i\frac{j_0j'}{k}} \cdot\\
  \nonumber
  &\int_0^1
  s(u'-\tfrac{j'}{k}+\tfrac{m}{k}, v'')e^{k\pi
    i(u'-\frac{j'}{k}+\frac{m}{k})v''} e^{2\pi ij'v''} e^{-2\pi imv'}
  dv'' \Biggr]\\
  \nonumber
  &= e^{-k\pi iu'v'}\sum_{j_0=0}^{k-1}\sum_{r=0}^{k-1}\sum_{l\in\setZ}
  e^{-2\pi iklv'}e^{-2\pi irv'} \Biggl[ \frac{1}{k}\sum_{j'=0}^{k-1}
  e^{2\pi i\frac{j_0j'}{k}}\\
  \nonumber
  &\int_0^1 s(u'+\tfrac{r-j'}{k}, v'') e^{\pi iklv''} e^{-2\pi
    i\frac{j_0r}{k}} e^{k\pi i(u'-\frac{r-j'}{k})v''} e^{k\pi i klv''}
  e^{2\pi ij'v''} dv'' \Biggr]\\
  \nonumber
  &= e^{-k\pi iu'v'} \sum_{r=0}^{k-1} \sum_{l\in\setZ} e^{-2\pi
    i(kl+r)v'} \int_0^1 s(u',v'') e^{k\pi iu'v''} e^{2\pi
    i(kl_r)v''}dv''\\
  \nonumber
  &= e^{-k\pi iu'v'} s(u',v') e^{k\pi iu'v'}\\
  \nonumber
  &= s(u',v')
\end{align}
Thus
\begin{align}
  \nonumber
  (W^{(k)})^{-1}\colon \cH \to S(\setR) \otimes \setC^k
\end{align}
is given by
\begin{align}
  \nonumber
  ((W^{(k)})^{-1}(s))_{j_0}(u') = & e^{-2\pi ij_0u'} \\
  \nonumber
  & \frac{1}{k}
  \sum_{j'=0}^{k-1} e^{2\pi i\frac{j_0j'}{k}} \int_0^1 s(u' -
  \tfrac{j'}{k}, v'') e^{k\pi i(u'-\frac{j'}{k})v''} e^{2\pi ijv''} dv''
\end{align}
\end{proof}

\begin{proposition}
We have the following description of the operators $U,V,\tilde U, \tilde V$ conjugated by the k'th order Weil-Gel'fand-Zak transform
\begin{align}
  \nonumber
    (W^{(k)})^{-1}U W^{(k)}(f_0,\ldots f_{k-1}) &= (e^{2\pi i k b^{-2} u'} f_{1}(u'), e^{2\pi i k b^{-2} u'} f_{2}(u'), \ldots, e^{2\pi i k b^{-2} u'} f_{0}(u'))\\
     \nonumber
    (W^{(k)})^{-1}V W^{(k)}(f_0,\ldots f_{k-1}) &= ( f_{0}(u' + k^{-1}),  f_{1}(u' + k^{-1}), \ldots,  f_{k-1}(u' + k^{-1}))\\
     \nonumber
(W^{(k)})^{-1}\tilde U W^{(k)}(f_0,\ldots f_{k-1}) &= (f_{0}(u' -\bar  b^{-2}-k^{-1}), e^{2\pi i  \bar b^{-2} } f_{1}(u' -\bar b^{-2}-k^{-1}), \\
 \nonumber
& \phantom{hjhhh} \ldots, e^{2\pi i  \bar b^{-2} (k-1)} f_{k-2}(u'-\bar b^{-2}-k^{-1}))\\
 \nonumber
 (W^{(k)})^{-1}\tilde V W^{(k)}(f_0,\ldots f_{k-1}) &= ( f_{1}(u'),  f_{2}(u'), \ldots,  f_{0}(u'))
    \end{align}
\end{proposition}

\section{The quantum representation of the Ptolemy groupoid}
\label{PtRep}

Recall from section \ref{MS}, that we need to consider the following birationale coordinate changes between the ratio coordinates, when we perform a Ptolemy flip of the triangulation, e.g. the $(u_1,v_2,u_2,v_2)$ on $(\setC^*)^{\times 2}$ are mapped to the coordinates $(w_1,z_2,\tilde w_2,\tilde z_2)$ on $(\setC^*)^{\times 2}$ by the birational map 
\begin{align}
  \nonumber
   w_1 &=  u_1 u_2,   & z_1 & =  u_1 v_2 + v_1 \\
  \nonumber
   w_2 &=  v_1 u_2 (u_1 v_2 + v_1)^{-1},  & z_2 &=  v_2(u_1 v_2 + v_1)^{-1} \\
  \end{align}
In each of the four sets of coordinates $(u_1,v_1)$, $(u_2,v_2)$, $(w_1,z_1)$ and $(w_2,z_2)$ we have the quantisation described in section \ref{GQ}, which yields the four Weil-pairs of operators
$$ (U_i,V_i), (W_i,Z_i), i = 1,2.$$ 
In is not hard to check that the above Ptolemy transformation of these coordinates preserves the holomorphic symplectic form induced from $\tilde \Omega$ in each pair.

Let us now consider the operator introduce by Faddeev jointly with the second author of this paper in connection with quantisation of decorated Teichm\"{u}ller theory in \cite{FaddevRinatQTeichm\"{u}ller}
\begin{align}
  \nonumber
  T_{b} = T_{b}(x_1,y_1,x_2,y_2) = \Psi_{b}(x_1 - y_1+y_2)e^{2\pi i y_1x_2}
\end{align}
where $\Psi_b$ is a solution to the functional equations
\begin{align}
  \nonumber
  \Psi_b(z+i\tfrac{b}{2}) &= \Psi_b(z-i\tfrac{b}{2})(1+e^{2\pi bz}),
  \quad z\in\setC, b\in\setC.
\end{align}
Here $(x_i,y_i)$ are operators with the commutations relations
$$ [x_i,y_i] = \frac{1}{2\pi i}$$
and
$$[x_1,x_2] = [x_1,y_2] = [y_1,x_2] = [y_1,y_2] = 0.$$
The assumption on $\Psi_b$ is such that $T_{b}$ can be defined in some operator algebra of unbounded operators containing $(x_i,y_i)$. We will return to this point below.
Under these assumptions we can derive a few commutation properties, which $T_b$ and the $(x_i,y_i)$ satisfies.  
\begin{proposition}
  \label{prop:2} We have the following relations
  \begin{enumerate}
  \item $T_b x_1 = (x_1+x_2)T_b $
  \item $T_b(y_1+x_2) = (y_1+x_2)T_b$
  \item $T_b(y_1+y_2) = y_2T_b$
  \item $T_be^{2\pi b y_1} = (e^{2\pi b(x_1+y_2)}+e^{2\pi b y_1})T_b$
  \end{enumerate}
  To prove these relations, we need the a couple of elementary formulae, which we repeat here for the convenience of the reader.
\end{proposition}
\begin{lemma}
  \label{lem:1}
  Suppose we have operators $x, y$ in an operator algebra, such that $z = [x,y]$ and $[z, x] =
  0$. Then
  \begin{align}
  \nonumber
    f(x)y &= yf(x) + zf'(x)\\
    \nonumber
    e^x f(x) &= f(y+z) e^x
  \end{align}
for every power series such that $f(x)$, $f'(x)$ and $f(y+z)$ can be defined in the same operator algebra.
\end{lemma}
\begin{proof}
  Let
  \begin{align}
  \nonumber
    f(x) &= \sum_{j=0}^{\infty} a_j x^j.
  \end{align}
  Then
  \begin{align}
  \nonumber
    [f(x), y] &= \sum_j a_j[x^j, y]\\
  \nonumber
    &= \sum_j a_j \sum_{k=0}^{j-1}x^k[x,y]x^{j-k-1}\\
  \nonumber
    &= \sum_j a_j j z x^{j-1}\\
  \nonumber
    &= zf'(x)
  \end{align}
Let us now establish the second equation
  \begin{align}
  \nonumber
    e^xy^l &= ye^xy^{l-1} + ze^x y^{l-1}\\
  \nonumber
    &= (y+z)e^x y^{l-1}\\
  \nonumber
    &= (y+z)^l e^x.
  \end{align}
  So
  \begin{align}
  \nonumber
    e^x f(y) &= e^x \sum_{j=0}^{\infty} a_j y^j \\
  \nonumber
    &= \sum_{j=0}^{\infty} a_j (y+z)^j e^x\\
  \nonumber
    &= f(y+z)e^x
  \end{align}
\end{proof}
\begin{proof}[Proof of Proposition]
 (1)
  \begin{align}
  \nonumber
    T_bx_1 &= \Psi_b(x_1 - y_1 + y_2) e^{2\pi iy_1 x_2} x_1\\
  \nonumber
    &= \Psi_b(x_1 - y_1 + y_2) x_1 e^{2\pi iy_1 x_2} +
    \Psi_b(x_1-y_1+y_2)x_2 e^{2\pi iy_1 x_2}\\
  \nonumber
    &= (x_1+x_2)T_b - \frac{1}{2\pi i}\Psi_b'(x_1-y_1+y_2) + \frac{1}{2\pi
    i}\Psi_b'(x_1-y_1+y_2)\\
  \nonumber
  &= (x_1+x_2)T_b
  \end{align}
  (2)
  \begin{align}
  \nonumber
    T_b(y_1+x_2) &= \Psi_b(x_1-y_1+y_2)e^{2\pi iy_1x_2}(y_1+x_2)\\
  \nonumber
    &= \Psi_b(x_1-y_1+y_2)(y_1+x+2)e^{2\pi iy_1x_2}\\
  \nonumber
    &= (y_1+x_2)\Psi_b(x_1-y_1+y_2)e^{2\pi iy_1x_2} \\
  \nonumber
    &\quad + \frac{1}{2\pi
      i}(-\Psi_b'(x_1-y_1+y_2)+\Psi_b'(x_1-y_1+y_2))e^{2\pi iy_1x_2}\\
  \nonumber
    &= (y_1+x_2)\Psi_b(x_1-y_1+y_2)e^{2\pi iy_1x_2}\\
  \nonumber
    &= (y_1+x_2)T_b
  \end{align}
  (3)
  \begin{align}
  \nonumber
    T_b(y_1+y_2) &= \Psi_b(x_1-y_1+y_2)e^{2\pi iy_1x_2}(y_1+y_2)\\
  \nonumber
    &= \Psi_b(x_1-y_1+y_2)(y_1+y_2)e^{2\pi iy_1x_2} -
    \Psi_b(x_1-y_1+y_2)y_1e^{2\pi iy_1x_2} \\
  \nonumber
    &= (y_1+y_2)T_b - y_1T_b + \frac{1}{2\pi i}(\Psi_b'(x_1-y_1+y_2) -
    \Psi_b(x_1-y_1+y_2))e^{2\pi iy_1x_2}\\
  \nonumber
    &= y_2T_b
  \end{align}
  (4)
  \begin{align}
  \nonumber
    T_be^{2\pi by_1} &=\Psi_b(x_1-y_1+y_2)e^{2\pi iy_1x_2}e^{2\pi by_1}\\
  \nonumber
    &= \Psi_b(x_1-y_1+y_2)e^{2\pi by_1}e^{2\pi iy_1x_2}\\
  \nonumber
    &= e^{2\pi by_1}\Psi_b(x_1-y_1+y_2-\frac{b}{i})e^{2\pi iy_1x_2}\\
  \nonumber
    &= e^{2\pi by_1}(1+e^{2\pi
      b(x_1-y_1+y_2+\frac{ib}{2})})\Psi_b(x_1-y_1+y_2)e^{2\pi iy_1x_2}\\
  \nonumber
    &= (e^{2\pi by_1} + e^{2\pi
      b(x_1+y_2-\frac{ib}{2}+\frac{ib}{2})})\Psi_b(x_1-y_1+y_2)e^{2\pi
      iy_1x_2}\\
  \nonumber
    &= (e^{2\pi by_1} + e^{2\pi b(x_1+y_2)})T_b
  \end{align}
\end{proof}

Consider the entire function
\begin{align}
  \nonumber
  (x;q)_\infty &= \prod_{i=0}^\infty (1-xq^i)
\end{align}

\begin{theorem}
For any pair of Weil operators $(x,y)$ in an operator algebra, for which we can define $(x;q)_\infty$ and $(y;q)_\infty$, we have the following operator identity
\begin{align}
  \nonumber
  (y;q)_\infty (x;q)_\infty = (x;q)_\infty (-yx;q)_\infty (y;q)_\infty
\end{align}
\end{theorem}

\begin{definition}
Let $\Psi_{b^{-1}}$ be the following entire function on the complex plane
$$ \Psi_{b^{-1}} (z) = (e^{2\pi(z + c_{b^{-1}})b^{-1}}, e^{i\pi b^{-2}})_\infty$$
\end{definition}

Standard domain theory for unbounded operators \cite{MR1009162,MR1009163} provides a dense subspace of $\cH'$ on which we can define the operator  $\Psi_{b^{-1}}(q_1 - p_1 + p_2)$ as an unbounded operator. We observe by Proposition \ref{spec}, that the closure of image under $\Psi_{\bar b^{-1}}$ of the spectrum of the operator $\tilde q_1 - \tilde p_1 + \tilde p_2$ does not contain zero, hence we can also define the inverse operator $\bar \Psi_{b^{-1}}(\tilde q_1 - \tilde p_1 + \tilde p_2)^{-1}$ on the same densely defined domain. 

\begin{definition}
We define the following unitary operator acting on $\cH'$
\begin{align}
  \nonumber
  \widetilde T &= \frac{\Psi_{b^{-1}}(q_1-p_1+p_2)}{\bar\Psi_{\bar b^{-1}}(\tilde
    q_1-\tilde p_1+\tilde p_2)} e^{2\pi i(p_1q_2+\tilde p_1\tilde q_2)}\\
  \nonumber
\end{align}
\end{definition}

The above comments on the independent definition of both the nominator and the denominator makes the proofs of following theorems very easy.

\begin{theorem}
  \label{thm:1}
  The operator
  \begin{align}
  \nonumber
  \widetilde T &= \frac{\Psi_b(u_1-v_1+v_2)}{\bar\Psi_{\bar b}(\bar
    u_1-\bar v_1+\bar v_2)} e^{2\pi i(v_1u_2+\bar v_1\bar u_2)}
  \end{align}
preserves the space $\cH$.
\end{theorem}

\proof

This follows directly from the commutation relations in Proposition \ref{prop:2} together with (\ref{nf}).
\eproof

\begin{theorem}
The unitary operator $\widetilde T$ satisfies the following pentagon relation
$$ \widetilde T_{12}\widetilde T_{13}\widetilde T_{23} = \widetilde T_{23}\widetilde T_{12}$$
\end{theorem}

\proof
Since we have $\tilde T$ written as the quotient of two expressions, which we know satisfies pentagon individually for insertion of self-adjoint restrictions of the $p$'s and $q$'s, we see by analyticity that the pentagon relation must also hold for $\widetilde T$. 
\eproof

Let ${\mathcal G}_{Pt}(\Sigma)$ be the Ptolemy groupoid of a surface $\Sigma$ of genus $g$ with $s>0$ punctures. Recall that the objects of ${\mathcal G}_{Pt}(\Sigma)$ are ideal triangulations of $\Sigma$ and morphism are flip transformations of ideal triangulations as first studied by Penner, who has also given a presentation of this groupoid \cite{Pennerbook}. A simple Euler characteristic computation shows that each ideal triangulation contains $4g+2s-4$ triangles.

\begin{theorem}\label{RepPt}
We get a unitary representation of ${\mathcal G}_{Pt}(\Sigma)$ for each $\Sigma$ as above, by assigning the Hilbert space $\cH^{\otimes (4g+2s-4)}$ to each object and by applying $\widetilde T : \cH \otimes \cH \ra \cH \otimes \cH$ in the two factors corresponding to a flip.
\end{theorem}

\section{Charged Tetrahedral operators and the symmetries}
\label{CPT}

In order to turn the unitary representation of the Ptolemy groupoid into a TQFT, we need to specify tetrahedral operators, which posses the symmetries of the tetrahedra, just as we did previously in \cite{AK1,AK2}. 
\begin{definition}
For arbitrary complex numbers $a,c$ we define
\begin{align}
  \nonumber
  \widetilde T(a,c) &= e^{-\pi ic_b^2(4(a-c)+1)/b} e^{+\pi
    ic_{\bar b}^2(4(\bar a-\bar c)+1)/b}\\
  \nonumber
 & e^{4\pi ic_b(cu_2-au_1)} e^{-4\pi ic_{\bar b}(\bar c\bar u_2-\bar
    au_1)}\widetilde T
  e^{-4\pi ic_b(av_2+cu_2)} e^{4\pi ic_{\bar b}(\bar a\bar v_2+\bar
    c\bar u_2)}
\end{align}
\end{definition}
In order to establish the needed symmetries of $\widetilde T(a,c)$ we introduce the following bilinear pairing
\begin{align}
  \nonumber
  (,)_F\colon \cH\times \cH \to\setC
\end{align}
given by
\begin{align}
  \nonumber
  (s_1,s_2)_F &= (s_1, Fs_2) = \int_{T^2}s_1\cdot Fs_2
\end{align}
where $F : C^\infty(T^2, L^k) \ra C^\infty(T^2, (L^*)^k)$ is given by 
$$F(s)(u',v') = s(v',u').$$
The welldefinedness of $F$ is shown by the simple computation
\begin{align}
  \nonumber
  F(s)(u'+1,v') &= s(v',u'+1) = e^{-\pi ikv'} s(v',u') = e^{-\pi ikv'}F(s)(u',v')\\
  \nonumber
  F(s)(u',v'+1) &= e^{\pi iku'}F(s)(u',v')
\end{align}
For any operator
\begin{align}
  \nonumber
  \tilde A\colon \mathcal{H}\to\mathcal{H}
\end{align}
we can define a bilinear pairing $\tilde A$ on $\cH$ by the formula
\begin{align}
  \nonumber
  \tilde A(s_1,s_2) = (s_1, \tilde As_2)_F = (s_1, F\tilde As_2).
\end{align}
We can introduce the transposed operator
$$ \tilde A^t : \cH \ra \cH$$
determined by
$$ (\tilde A^t s_1,s_2)_F = (s_1,\tilde As_2)_F.$$
We observe that
$$F\tilde A = \tilde A^tF$$
if an only if the bilinear pairing $\tilde A$ is symmetric,
since
\begin{align}
  \nonumber
  (s_1,s_2)_F &= \int_{T^2} s_1\cdot Fs_2 = \int_{T^2}^{} F(s_1\cdot Fs_2)\\
  \nonumber
  &= \int_{T^2}^{} Fs_1\cdot s_2 = (s_2,s_1)_F.
\end{align}

Let us use the bilinear pairing to define
\begin{align}
  \nonumber
  \widetilde T(a,c)\in 
  (\mathcal{H}\otimes\mathcal{H}\otimes\mathcal{H}\otimes\mathcal{H})^*
\end{align}
by the formula
\begin{align}
  \nonumber
  \widetilde T(a, c)(s_0\otimes s_2\otimes s_1\otimes s_3) &= (T(a,c)s_0\otimes
  s_2, s_1\otimes s_3)_F
\end{align}

The required tetrahedral symmetries (see the Fundamental Lemma 1 in \cite{MR3227503}) of the charged operator $\widetilde T(a,c)$ are the following operator relations
\begin{align}
  \nonumber
  \widetilde T(a,c)\circ \tilde A_3\circ \tilde A_2 &= \widetilde T
  (a,b)^{-1}P_{23}P_{01}P_{23}\\
  \nonumber
  \widetilde T(a,c)\circ \tilde B_3\circ \tilde A_0^* &= \widetilde T(b,c)^{-1}
  P_{02}P_{13}P_{21}P_{01}P_{23}\\
  \nonumber
  \widetilde T(a,c)\circ \tilde B_1\circ \tilde B_0^* &= \widetilde T(a,b)^{-1}
  P_{01}P_{01}P_{23}
\end{align}
where 
\begin{align}
  \nonumber
  \tilde A\colon \mathcal{H}\to\mathcal{H},   \tilde B\colon \mathcal{H}\to\mathcal{H}
\end{align}
are certain operators corresponding to symmetric bilinear parings 
$$\tilde A,\tilde B : \cH \otimes \cH \ra \setC.$$
\begin{theorem}
There exist unitary operators
\begin{align}
  \nonumber
  \tilde A\colon \mathcal{H}\to\mathcal{H},   \tilde B\colon \mathcal{H}\to\mathcal{H}
\end{align}
such that 
\begin{align}
  \nonumber
  \widetilde T(a,c)\circ \tilde A_3\circ \tilde A_2^* &= \widetilde T(a,b)^{-1}
  P_{01}\\
  \nonumber
  \widetilde T(a,c)\circ \tilde B_3\circ \tilde A_0^* &= \widetilde T(b,c)^{-1}
  P_{02}\\
  \nonumber
  \widetilde T(a,c)\circ \tilde B_1\circ \tilde B_0^* &= \widetilde T(a,b)^{-1}
  P_{23}.
\end{align}
\end{theorem}

\proof
By inserting the definition of $\widetilde T(a,c)$ into these equations, we get the corresponding equations
\begin{align}
  \label{eq:22}
  (T(a,c)(s_0\otimes \tilde A^*(s_2)), s_1\otimes \tilde A(s_3))_F &=
  (T(a,b)^{-1} (s_1\otimes s_2), s_0\otimes s_3)_F\\
  \nonumber
  \nonumber
  (T(a,c)(\tilde A^*(s_0)\otimes s_2), s_1\otimes \tilde B(s_3))_F &= (T(b,c)^{-1}(s_2\otimes
  s_0), s_1\otimes s_3)_F\\
  \nonumber
  \nonumber
  (T(a,c)(\tilde B^*(s_0)\otimes s_2), \tilde B(s_1)\otimes s_3)_F &= (T(a,b)^{-1}(s_0\otimes
  s_3), s_1\otimes s_2)_F
\end{align}
It is an easy verification to see that if we let
$$ \tilde A = ST, \tilde B = TS$$
where 
$T(s)(u',v') = s(u'+v',v') e^{\pi i k v'}$
then they verify these three equations.

\eproof

Now that we have build the charges tetrahedral operator $\widetilde T(a,c)$ with the required symmetries, we may proceed exactly like in our previous papers \cite{MR3227503} to define the associated TQFT, which in this case will result in quantum Chern-Simons theory for $PSL(2,\setC)$ at level $k$. Below we spell this out in detail in the alternative formulation of LCA's, first in a much more general setting and then we return to the example at hand in transformed by the $k$'th order Weil-Gel'fand-Zak transform, so that it becomes the theory associated to the LCA ${\mathbb R} \times \setZ_k$. 

\section{General framework}

\label{GF}

We let $\TORUS$ denote the complex unit circle. 
Let $A$ be a Pontryagin self-dual locally compact abelian (LCA) group. We denote by $\bkt{x;y}$ the symmetric bi-character which we will call \emph{Fourier kernel}, and we choose a Haar measure $\operatorname{d}\!x$ on $A$ so that the Fourier transformation $\mathsf{F}\colon L^2(A)\to L^2(A)$ defined by the formula
\begin{equation}
\mathsf{F}(f)(x)=\int_Af(y)\bkt{x;y} \operatorname{d}\!x
\end{equation}
 is an isometry, see e.g. \cite{standard ref on Fourrier analysis on LCA's}. 

\begin{definition}
A  function $\bkt{x}$ satisfying the relations
\begin{equation}
\bkt{x}=\bkt{-x},\quad \bkt{x;y}=\frac{\bkt{x+y}}{\bkt{x}\bkt{y}}.
\end{equation}
is called a \emph{Gaussian exponential}.
\end{definition}

\begin{example}
 If $A=\REALS$ with standard Lebesgue measure, then we have $\bkt{x;y}=e^{2\pi\imun xy}$ and $\bkt{x}=e^{\pi\imun x^2}$.
\end{example}
Assume from now on that a Gaussian exponential has been fixed.
It is clear that $\bkt{0}=1$, and  we also have $\bkt{x;x}=\bkt{x}^2$ as the consequence of the following calculation:
\begin{equation}
\bkt{x;x}=\frac{1}{\bkt{x;-x}}=\frac{\bkt{x}\bkt{-x}}{\bkt{x-x}}=\frac{\bkt{x}\bkt{-x}}{\bkt{0}}=\bkt{x}^2.
\end{equation}
Evenmore, one can easily prove by recurrence that $\bkt{nx}=\bkt{x}^{n^2}$ for any integer $n$.
\begin{definition}
A group homorphism $\mathsf{f}$  from $A$ to the group of unitary operators acting on $L^2(A)$ will be called a \emph{u-operator}, and we will write 
\begin{equation}
\mathsf{f}\colon A\to U(L^2(A)),\quad x\mapsto \bkt{\mathsf{f};x}.
\end{equation}
\end{definition}
We can consider integer multiples of a u-operator, through the formula
\begin{equation}
\bkt{n\mathsf{f};x}=\bkt{\mathsf{f};x}^n.
\end{equation}
This can be generalized to a slightly non-commuting setting as follows.
\begin{definition}
We say that two u-operators $\mathsf{p}$ and $\mathsf{q}$ form a \emph{Heisenberg pair}  if they satisfy the relation 
\begin{equation}\label{heis}
\bkt{\mathsf{p};x}\bkt{\mathsf{q};y}=\bkt{x,y}^{\ell(\mathsf{p},\mathsf{q})}\bkt{\mathsf{q};x}\bkt{\mathsf{p};y}, \quad\forall x,y\in A\end{equation}
for some integer $\ell(\mathsf{p},\mathsf{q})$, which we will call the \emph{level} of the pair $(\mathsf{p},\mathsf{q})$. 
\end{definition}

Any u-operator forms a Heisenberg pair of level $0$ with itself. A typical example of a Heisenberg pair of level $1$ is given by the u-operators defined by the relations
\begin{equation}
(\bkt{\mathsf{p};x}f)(y)=f(x+y),\quad (\bkt{\mathsf{q};x}f)(y)=\bkt{y;x}f(y),
\end{equation}
which satisfy relation~\eqref{heis} with $\ell(\mathsf{p},\mathsf{q})=1$ as the consequence of the following simple calculation:
\begin{equation}
(\bkt{\mathsf{p};x}\bkt{\mathsf{q};y}f)(z)=(\bkt{\mathsf{q};y}f)(x+z)=\bkt{y;x+z}f(x+z)
\end{equation}
and
\begin{equation}
(\bkt{\mathsf{q};y}\bkt{\mathsf{p};x}f)(z)=\bkt{z;y}(\bkt{\mathsf{p};x}f)(z)=\bkt{z;y}f(x+z).
\end{equation} 
This particular example will be called the \emph{canonical Heisenberg pair}.
\begin{definition}
For any Heisenberg pair $\mathsf{p}$ and $\mathsf{q}$ of level $n$, we define their \emph{sum} $\mathsf{p}+\mathsf{q}$
by the formula
\begin{equation}
\bkt{\mathsf{p}+\mathsf{q};x}=\bkt{x}^n\bkt{\mathsf{q};x}\bkt{\mathsf{p};x}.
\end{equation}
\end{definition}
The notation and term are justified by the property
\begin{equation}
\bkt{\mathsf{p}+\mathsf{q};x}\bkt{\mathsf{p}+\mathsf{q};y}=\bkt{\mathsf{p}+\mathsf{q};x+y}
\end{equation}
as the consequence of the following calculation:
\begin{multline}
 \bkt{\mathsf{p}+\mathsf{q};x}\bkt{\mathsf{p}+\mathsf{q};y}=\bkt{x}^n\bkt{\mathsf{q};x}\bkt{\mathsf{p};x}\bkt{y}^n\bkt{\mathsf{q};y}\bkt{\mathsf{p};y}\\=\bkt{x}^n\bkt{y}^n\bkt{x;y}^n\bkt{\mathsf{q};x}\bkt{\mathsf{q};y}\bkt{\mathsf{p};x}\bkt{\mathsf{p};y}
 =\bkt{x+y}^n\bkt{\mathsf{q};x+y}\bkt{\mathsf{p};x+y}.
\end{multline}
For the moment this is only a binary operation in the set of u-operators, but it is commutative  as it is seen from the obvious equality
\begin{equation}
\bkt{x}^n\bkt{\mathsf{q};x}\bkt{\mathsf{p};x}=\bkt{x}^{-n}\bkt{\mathsf{p};x}\bkt{\mathsf{q};x},
\end{equation}
while the associativity is more lengthy to check but straightforward. This construction implies that any set of u-operators where any pair of elements forms a Heisenberg pair generates an abelian group, and the level becomes an antisymmetric bilinear form on that group. Such a group will be called a \emph{Heisenberg group of u-operators}, if the level bilinear form on it is non-trivial.

Suppose that $f$ is a complex valued function  on $A$ and  $\mathsf{g}$ a u-operator, such that 
the operator  $f(\mathsf{g})$ can be defined on an open dense subspace of $L^2(A)$ by the formula:
\begin{equation}\label{opfun}
f(\mathsf{g})=\int_A\mathsf{F}^{-1}(f)(x)\bkt{\mathsf{g};x}\operatorname{d}\!x.
\end{equation}
Under certain conditions, a property of $f$ to be bounded/unitary will imply that the corresponding operator is bounded/unitary and further, the operators associated to two functions always commute and their product is the operator associated to the product of the functions. We will see this in the examples we do in detail below. 

The cornerstone of our work is the following notion.
\begin{definition}
 A \emph{quantum dilogarithm} over $A$ is a function $\phi\colon A\to\TORUS$  which satisfies the inversion relation
 \begin{equation}
\phi(x)\phi(-x)=\bkt{x}\phi(0)^2,\quad \forall x\in A,
\end{equation}
and the following Faddeev's five term operator relation is satisfied
 \begin{equation}\label{pent}
\phi(\mathsf{p})\phi(\mathsf{q})=\phi(\mathsf{q})\phi(\mathsf{p}+\mathsf{q})\phi(\mathsf{p})
\end{equation}
for some Heisenberg pair of u-operators $\mathsf{p}$ and $\mathsf{q}$ with $\ell( \mathsf{p}, \mathsf{q})=1$,
\end{definition}
 The inversion relation is consistent with Faddeev's relation in the following sense. Denote $\phi(x)\phi(-x)\equiv \sigma(x)$. Remarking that $\ell(\mathsf{q},-\mathsf{p})=1$, we multiply \eqref{pent} from the right by $\phi(-\mathsf{p})$ and use Faddeev's relation on the left hand side. The result is as follows:
 \begin{equation}
\sigma(\mathsf{p})\phi(\mathsf{q}-\mathsf{p})\phi(\mathsf{q})=\phi(\mathsf{q})\phi(\mathsf{q}+\mathsf{p})\sigma(\mathsf{p})
\end{equation}
which becomes an identity for any $\phi$, for which these operators are defined, provided $\sigma(x)$ is a scalar multiple of $\bkt{x}$.

At the basis of our approach of constructing quantum dilogarithms is the following five-term identity in the algebra of formal power series of two $q$-commuting formal variables:
\begin{equation}
(\mathsf{v};q)_\infty(\mathsf{u};q)_\infty=(\mathsf{u};q)_\infty(-\mathsf{v}\mathsf{u};q)_\infty(\mathsf{v};q)_\infty,\quad \mathsf{u}\mathsf{v}=q\mathsf{v}\mathsf{u},
\end{equation}
where
\begin{equation}
(x;q)_\infty\equiv (1-x)(1-xq)(1-xq^2)\dots
\end{equation}
This identity first has been discovered in \cite{MR1264393}. The challenge is to convert it into the operator identity \eqref{pent} for unitary operators in the Hilbert space $L^2(A)$.

Let $q$ be a complex number inside the unit disk, and $\chi\colon A\to\COMPLEXS^\times$ a continuous group homomorphism such that the function $(-q^{1/2}\chi(x);q)_\infty$ does not vanish on $A$, i.e. $\chi(x)q^{\frac12+n}\ne-1$ for all $x\in A$ and $n\in\INTEGERS_{\ge0}$, and for a Heisenberg pair of u-operators $\mathsf{p}$ and $\mathsf{q}$ of level $1$, the operators $\chi(\mathsf{q}),\chi(\mathsf{p}),\bar\chi(\mathsf{q}),\bar\chi(\mathsf{p})$ are defined and satisfy the following relations:
\begin{equation}
\chi(\mathsf{q})\chi(\mathsf{p})=q\chi(\mathsf{p})\chi(\mathsf{q}),\quad 
\chi(\mathsf{q})\bar\chi(\mathsf{p})=\bar\chi(\mathsf{p})\chi(\mathsf{q}),
\end{equation}
where $\bar\chi$ is the complex conjugate function. Then, under some further conditions on the objects entering the construction, we expect that the 
 function 
\begin{equation}\label{mfun}
\Phi_{q,\chi}\colon A\to \TORUS, \quad \Phi_{q,\chi}(x)=e^{2\imun\arg\left(-q^{1/2}\chi(x);q\right)_\infty},
\end{equation}
is a quantum dilogarithm over $A$. 

The first example of a quantum dilogarithm over the LCA group $\REALS$ was found by L.~Faddeev in \cite{Faddeev1994}. It has the form
\begin{equation}\label{ffun}
\phi(x)=\fad_\theta(x)\equiv\Phi_{q,\chi}(x)
\end{equation}
 with 
 \begin{equation}
q=e^{2\pi\imun \theta^2},\quad \chi(x)=e^{2\pi \theta x},\quad |\theta|=1,\quad \Re \theta>0,\ \Im\theta>0.
\end{equation}
Faddeev's quantum dilogarithm is closely related to Shintani's double sine function \cite{MR0460283,MR1105522,Barnes1904}, but the five term relation~\eqref{pent} seems not to be known before Faddeev's paper \cite{Faddeev1994}. For further properties of Faddeev's quantum dilogarithm see, for example, \cite{MR1828812,MR2171695,MR1770545,MR2952777}.

In this paper, we generalize this example to the case of the LCA group $\REALS\oplus(\INTEGERS/\myN\INTEGERS)$ for any positive integer $\myN$. The solution is again of the form 
\begin{equation}\label{dfun}
\phi(x,n)=\gfad_\theta(x,n)\equiv\Phi_{q,\chi}(x,n)
\end{equation}
 with
 \begin{equation}
q=e^{2\pi\imun (1+\theta^2)/\myN},\quad \chi(x,n)=e^{2\pi \theta x/\sqrt{\myN}}e^{2\pi\imun n/\myN},\quad |\theta|=1,\quad \Re \theta>0, \ \Im\theta>0.
\end{equation}

\section{Functions of the Faddeev type} 
\label{FoFT}

This section, essentially reproduces the constructions of Section~4 of the work \cite{Kashaev2014} by adapting the formulas to the context of arbitrary self-dual LCA group $A$ where all elements are divisible by 2. This adaptation boils down to replacements of all $\REALS$'s by $A$'s, the operators $e^{2\pi\imun\mathsf{f}x}$ by $\bkt{\mathsf{f};x}'s$, and the Gaussian exponentials
$e^{\pi\imun x^2}$ by $\bkt{x}$. Nonetheless, we think that it is instructive to review here in detail all these constructions.

 For any non-negative integer $n$, we let $[n]$ denote the set $\Z_{\ge0}\cap\Z_{\le n}$ and $\Delta[n]$ denote  the standard combinatorial $n$-simplex seen, for example, as the power set $2^{[n]}$. We also denote by $\Delta[n]_i$ the set of $i$-dimensional simplexes of $\Delta[n]$. A function
\begin{equation}
f\colon[4]\times A\to\COMPLEXS,\quad (j,x)\mapsto f_j(x),
\end{equation}
is called of \emph{the Faddeev type} over $A$ if it is bounded and satisfies the operator relation
\begin{equation}\label{eq35}
f_1(\mathsf{p})f_3(\mathsf{q})=f_4(\mathsf{q})f_2(\mathsf{p}+\mathsf{q})f_0(\mathsf{p})
\end{equation}
where $\mathsf{p}$ and $\mathsf{q}$ form a Heisenberg pair of u-operators with $\ell(\mathsf{p},\mathsf{q})=1$. The difference with a quantum dilogarithm over $A$ is that we admit five different functions on $A$, we do not assume unitarity, and we do not impose the inversion relation. That means, that the constant unit function 
 \begin{equation}
f_j(x)=1,\quad \forall j\in[4],
\end{equation}
constitutes a trivial example of a function of the Faddeev type.
\begin{lemma}
 A square integrable function $f$ is of the Faddeev type over A if and only if
 \begin{equation}\label{eq50}
\tilde f_1(x)\tilde f_3(y)=\bkt{x;-y}\int_A \tilde f_4(y-z)\tilde f_2(z)\tilde f_0(x-z)\bkt{z}dz,\quad\forall (x,y)\in A^2,
\end{equation}
where
\begin{equation}
\tilde f(x)\equiv(\mathsf{F}^{-1}f)(x)= \int_A\bkt{x;-y}f(y)dy.
\end{equation}
\end{lemma}
\begin{proof}
By using the definition~\eqref{opfun}, equality~\eqref{eq35} takes the form of an operator valued  integral equality
\begin{multline}
\int_{A^2}\tilde f_1(x)\tilde f_3(y)\bkt{\mathsf{p};x}\bkt{\mathsf{q};y}\operatorname{d}\!x\operatorname{d}\!y\\=\int_{A^3}\tilde f_4(y)\tilde f_2(z)\tilde f_0(x)\bkt{\mathsf{q};y}\bkt{\mathsf{p}+\mathsf{q};z}\bkt{\mathsf{p};x}\operatorname{d}\!x\operatorname{d}\!y\operatorname{d}\!z
\end{multline}
which, by using the properties of Heisenberg pairs
takes the form
\begin{multline}\label{eq40}
\int_{A^2}\tilde f_1(x)\tilde f_3(y)\bkt{x;y}\bkt{\mathsf{q};y}\bkt{\mathsf{p};x}
\operatorname{d}\!x\operatorname{d}\!y\\=\int_{A^3}\tilde f_4(y)\tilde f_2(z)\tilde f_0(x)\bkt{z}\bkt{\mathsf{q};y+z}\bkt{\mathsf{p};x+z}\operatorname{d}\!x\operatorname{d}\!y\operatorname{d}\!z\\
=\int_{A^3}\tilde f_4(y-z)\tilde f_2(z)\tilde f_0(x-z)\bkt{z}\bkt{\mathsf{q};y}\bkt{\mathsf{p};x}\operatorname{d}\!x\operatorname{d}\!y\operatorname{d}\!z.
\end{multline}
Comparing the coefficients of the operators $\bkt{\mathsf{q};y}\bkt{\mathsf{p};x}$, we come to the conclusion that  equality~\eqref{eq40} is true if and only if equality~\eqref{eq50} is true.
\end{proof}
Taking the complex conjugate of \eqref{eq50}, we also have
\begin{equation}\label{eq60}
\bar{\tilde f}_1(x)\bar{\tilde f}_3(y)
=\bkt{x;y}\int_A \bar{\tilde f}_4(y-z)\bar{\tilde f}_2(z)\bar{\tilde f}_0(x-z)\overline{\bkt{z}}\operatorname{d}\!z,\quad\forall (x,y)\in A^2.
\end{equation}
\begin{definition}
 We say that a function
\begin{equation}\label{eq-10}
\phi\colon [4]\times A^2\to\COMPLEXS,\quad (j,x,y)\mapsto \phi_j(x,y)
\end{equation}
is of the \emph{beta pentagon type} over $A$ if the following five term integral relation is satisfied:
\begin{equation}\label{eq00}
  \phi_1(x,y)\phi_3(u,v) =\int_{A}
 \phi_4(u+y,v- z) 
 \phi_2(x+y+u+v-z,z)  \phi_0(x+v,y-z)\operatorname{d}\!z.
\end{equation}
The relation~\eqref{eq00} itself will be called the \emph{beta pentagon relation} over $A$.
\end{definition}
The motivation for the beta pentagon relation comes from the following combinatorial interpretation. 
Given a map as in \eqref{eq-10}, we define another map
\begin{equation}
W\colon \Delta[4]_3\times A^{\Delta[3]_1}\to \COMPLEXS
\end{equation}
by assigning
\begin{equation}
W(\partial_i\Delta[4],x)=\phi_i(x_{01}+x_{23}- x_{03}-x_{12},x_{03}+x_{12}-x_{02}-x_{13})
\end{equation}
where $x_{jk}$ is the $x$-image of the edge $\{j,k\}$. Then, the equality 
\begin{equation}
\prod_{i\in\{1,3\}}W(\partial_i\Delta[4],\varepsilon_i^{*}x)=
\int_{A}\operatorname{d}\!x_{13}\prod_{i\in\{0,2,4\}}W(\partial_i\Delta[4],\varepsilon_i^{*}x),\quad \forall x\in A^{\Delta[4]_1},
\end{equation}
with the standard injections
\(
\varepsilon_i\colon [3]\to [4],\ i\in[4],
\)
defined by
\begin{equation}
\varepsilon_i(j)=\left\{
\begin{array}{cl}
 j&\mathrm{if}\ j<i;\\
 j+1&\mathrm{otherwise},
\end{array}
\right.
\end{equation}
is a consequence of the beta pentagon relation for  $\phi$.
\begin{theorem}\label{theorem}
 Let $f$ and $g$ be two functions of the Faddeev type over a Pontryagin self-dual LCA group $A$ where any element is divisible by $2$. Then the function 
 \begin{equation}\label{eq70}
\varphi\colon [4]\times A^2\to\COMPLEXS,\quad
(j,x,y)\mapsto \varphi_j(x,y)=\int_A \bkt{y;t}f_j\left(t+\frac x2\right)\bar{g}_j\left(t-\frac x2\right)\operatorname{d}\!t
\end{equation}
is of the beta pentagon type over $A$, i.e. it satisfies relation~\eqref{eq00}.
\end{theorem}
It is instructive to give an operator interpretation for the formula in \eqref{eq70}. 
We have the following relations 
\begin{equation}\label{eq72}
\bkt{\mathsf{p};x}\mathsf{F}=\mathsf{F}\bkt{\mathsf{q};x},\quad \bkt{\mathsf{q};x}\mathsf{F}=\mathsf{F}\bkt{\mathsf{p};-x},\quad \forall x\in A.
\end{equation}
where $\mathsf{p}$ and $\mathsf{q}$ is the canonical Heisenberg pair:
\begin{equation}
(\bkt{\mathsf{p};x}f)(y)=f(x+y),\quad (\bkt{\mathsf{q};x}f)(y)=\bkt{x;y}f(y).
\end{equation}
By using Dirac's bra-ket notation for the scalar product in $L^2(A)$:
\begin{equation}
\langle f\vert g\rangle\equiv\int_A \bar f(x) g(x)\operatorname{d}\!x,
\end{equation}
for any $(x,y)\in A^2$, we have
\begin{multline}\label{eq73}
\varphi_j(x,y)\equiv\int_A\bkt{t;y} f_j\left(t+\frac x2\right)\bar g_j\left(t-\frac x2\right)\operatorname{d}\!t\\
=\int_{A}(\overline{\bkt{\mathsf{p};-x/2}g_j)}(t)\bkt{t;y}\left( \bkt{\mathsf{p};x/2}f_j\right)(t)\operatorname{d}\!t\\
=\int_{A}(\overline{\bkt{\mathsf{p};-x/2}g_j)}(t)\left(\bkt{\mathsf{q};y} \bkt{\mathsf{p};x/2}f_j\right)(t)\operatorname{d}\!t\\
=\langle\bkt{\mathsf{p};-x/2}g_j|\bkt{\mathsf{q};y} \bkt{\mathsf{p};x/2}f_j\rangle
=\langle g_j|\bkt{\mathsf{p};x/2}\bkt{\mathsf{q};y} \bkt{\mathsf{p};x/2}f_j\rangle\\
=\bkt{x/2;y}\langle g_j|\bkt{\mathsf{q};y} \bkt{\mathsf{p};x}f_j\rangle.
\end{multline}
This formula allows us to easily prove the following lemma.
\begin{lemma}
Let $\varphi_j(x,y)$ be defined as in \eqref{eq70}. Then the following equality holds true
 \begin{equation}\label{eq75}
\varphi_j(x,y)=\int_A \bkt{x;t}\tilde f_j\left(t-\frac y2\right)\bar{\tilde g}_j\left(t+\frac y2\right)\operatorname{d}\!t,
\end{equation}
where 
\begin{equation}
\tilde f\equiv\mathsf{F}^{-1}f,\quad \forall f\in L^2(A).
\end{equation}
\end{lemma}
\begin{proof}
Indeed, by using  \eqref{eq73} and \eqref{eq72}, we have
\begin{multline}
\varphi_j(x,y)\bkt{x/2;-y}=\langle g_j|\bkt{\mathsf{q};y} \bkt{\mathsf{p};x}f_j\rangle=
\langle g_j|\bkt{\mathsf{q};y} \bkt{\mathsf{p};x}\mathsf{F}\mathsf{F}^{-1}f_j\rangle\\=
\langle g_j|\mathsf{F}\bkt{\mathsf{p};-y} \bkt{\mathsf{q};x}\mathsf{F}^{-1}f_j\rangle=
\langle \mathsf{F}^{-1}g_j\vert \bkt{\mathsf{p};-y} \bkt{\mathsf{q};x}\mathsf{F}^{-1}f_j\rangle\\\equiv
\langle \tilde g_j\vert \bkt{\mathsf{p};-y} \bkt{\mathsf{q};x}\tilde f_j\rangle,
\end{multline}
and formula~\eqref{eq75} now follows by applying \eqref{eq73} backwards with $f_j$ and $g_j$ replaced with $\tilde f_j$ and $\tilde g_j$.
\end{proof}
\begin{proof}[Proof of Theorem~\ref{theorem}] By using \eqref{eq75}, we write
\begin{multline*}
 \varphi_1(x,y)\varphi_3(u,v)\\
 =\int_{A^2}
 \bkt{x;s} \bkt{u;t}\tilde f_1\left(s-\frac y2\right) \tilde f_3\left(t-\frac v2\right)\bar{\tilde g}_1\left(s+\frac y2\right)
\bar{\tilde g}_3\left(t+\frac v2\right)\operatorname{d}\!s\operatorname{d}\!t\\
\end{multline*}
and then continue by applying \eqref{eq50} and \eqref{eq60}
\begin{multline*}
 =\int_{A^4}
  \bkt{x;s} \bkt{u;t} \bkt{v;s} \bkt{y;t} \bkt{z} \overline{\bkt{w}} \tilde f_4\left(t-\frac v2-z\right)\bar{\tilde g}_4\left(t+\frac v2-w\right)\\
 \times\tilde f_2(z) \bar{\tilde g}_2(w)
 \tilde f_0\left(s-\frac y2-z\right)
\bar{\tilde g}_0\left(s+\frac y2-w\right)
\operatorname{d}\!s\operatorname{d}\!t\operatorname{d}\!z\operatorname{d}\!w\\
\end{multline*}
changing the variables of integration $s\mapsto s+(z+w)/2$, $t\mapsto t+(z+w)/2$, the integrations over $s$ and $t$ can be performed by using \eqref{eq75}
\begin{multline*}
 =\int_{A^2}
 \bkt{x+y+u+v+z-w;(z+w)/2} 
 \varphi_4(y+u,v+z-w) \\
 \times\tilde f_2(z) \bar{\tilde g}_2(w)\varphi_0(x+v,y+z-w)
\operatorname{d}\!z\operatorname{d}\!w\\
 =\int_{A^2}
\bkt{x+y+u+v+z;w+z/2}  
 \varphi_4(y+u,v+z) \tilde f_2(z+w) \bar{\tilde g}_2(w)\varphi_0(x+v,y+z)
\operatorname{d}\!z\operatorname{d}\!w\\
\end{multline*}
where we have shifted $z\mapsto z+w$. Finally, by shifting $w\mapsto w-z/2$ and using again \eqref{eq75}, we obtain
\begin{multline*}
 =\int_{A^2}
 \bkt{x+y+u+v+z;w}
 \varphi_4(y+u,v+z) \tilde f_2\left(w+\frac z2\right) \bar{\tilde g}_2\left(w-\frac z2\right)\varphi_0(x+v,y+z)\\
\times \operatorname{d}\!z\operatorname{d}\!w
 =\int_{A}
 \varphi_4(y+u,v+z) 
 \varphi_2(x+y+u+v+z,-z)  \varphi_0(x+v,y+z)
\operatorname{d}\!z.
\end{multline*}
\end{proof}
\subsection{A special case with $g_j(x)=1$}
In the special case when $g_j(x)=1$, we obtain from \eqref{eq70}
\begin{equation}
\varphi_j(x,y)=\int_A \bkt{y;t}f_j\left(t+x/2\right)\operatorname{d}\!t=\int_A \bkt{y;t-x/2}f_j\left(t\right)\operatorname{d}\!t=\bkt{x;-y/2}\tilde f_j\left(-y\right),
\end{equation}
which is a solution quasiperiodic in variable $x$:
\begin{equation}
\varphi_j(x+b,y)=\varphi_j(x,y)\bkt{y/2;-b}.
\end{equation}
According to \cite{Kashaev2014} (Section~3 thereof)  this function is of the automorphic beta pentagon  type $(B,0,h)$ with $h_y(b)=\bkt{y/2;-b}$ if and only if $\bkt{b;b'}=1$ for all $b,b'\in B$. In that case, Theorem~1 of  \cite{Kashaev2014} implies that the function
 \begin{multline}
\psi_j(x,y)=\int_B\varphi_j(x,y+b)\bkt{\mu_j-(b+x)/2;b}\operatorname{d}\!b\\
=\bkt{x;-y/2}\int_B(\mathsf{F}f_j)(y+b)\overline{\bkt{b}}\bkt{b;\mu_j-x}\operatorname{d}\!b,\quad\forall (j,x,y)\in[4]\times A^2,
\end{multline}
where
\begin{equation}\label{eq:mui}
\mu_0=\mu_1=\alpha,\ \mu_2=\alpha+\beta,\ \mu_3=\mu_4=\beta,
\end{equation}
$\alpha$ and $\beta$ being arbitrary elements of $A$,
satisfies the automorphicity relations
\begin{multline}\label{eq:autopsi}
\psi_j(x+b,y)=\bkt{y/2;-b}\psi_j(x,y),\\
\psi_j(x,y+b)=\overline{\bkt{b}}\bkt{-\mu_j+x/2;b}\psi_j(x,y),\quad\forall (j,b,x,y)\in[4]\times B\times A^2,
\end{multline}
and the five term integral identity
\begin{equation}\label{eq11}
  \psi_1(x,y)\psi_3(u,v) =\int_{A/B}
 \psi_4(u+y,v-z) 
 \psi_2(x+y+u+v- z,z)  \psi_0(x+v,y- z)\operatorname{d}\!z
\end{equation}
where the domain of integration indicates that the integrand is invariant under $B$-shifts of the  integration variable.
\begin{example}
 Fix a positive integer $\myN$ and choose for $A$  the LCA group $\REALS\oplus\Z/\myN\Z$.  It is Pontryagin self-dual, and any element is divisible by 2 if $\myN$ is odd. In that case, the subgroup 
 $B=(\myN^{-1/2},1)\Z$ is dual to the quotient group $A/B\simeq \TORUS$.
\end{example}

\subsection{Fourier transformation formula for a quantum dilogarithm}

\begin{theorem}
 Let $\phi\colon A\to \TORUS$ be a quantum dilogarithm over a self-dual LCA group $A$ such that $(\mathsf{F}\phi)(x)^{\pm1}\ne 0$ almost everywhere. Then there exists a continuous group homomorphism $g\colon A\to\REALS^\times$ and an element of order two $\varepsilon\in A$ such that the one parameter family of integrals
 \begin{equation}
\phi_t(y)\equiv \int_A(\mathsf{F}^{-1}\phi)(x)|g(x)|^{t}\bkt{y;x}\operatorname{d}\! x,\quad t\in [0,1],
\end{equation}
which absolutely converge for $t\in]0,1[$, continuously interpolates between $\phi(y)=\phi_0(y)$ and 
\begin{equation}
\phi_1(y)= \frac{\bkt{\varepsilon-y}\tilde\phi(\varepsilon-y)}{\gamma\phi(0)^{2}},\quad \gamma\equiv\int_A\bkt{x}\operatorname{d}\! x
\end{equation}
\end{theorem}
\begin{proof} 
 We start from the following integral identity equivalent to the five-term Faddeev relation:
 \begin{equation}
\tilde \phi(x)\tilde \phi (y)=\bkt{x,-y}\int_A \tilde \phi(y-z)\tilde \phi (z)\tilde \phi(x-z)\bkt{z}\operatorname{d}\! z,\quad \tilde\phi\equiv \mathsf{F}^{-1}\phi,
\end{equation}
and apply the Fourier transformation on variable $y$ to both sides of the equality, i.e. we multiply by $\bkt{y;u+x}$ and integrate over $y$:
 \begin{multline}
\tilde \phi(x)\phi (x+u)=\int_{A^2} \bkt{u,y}\tilde \phi(y-z)\tilde \phi (z)\tilde \phi(x-z)\bkt{z}\operatorname{d}\! y\operatorname{d}\! z\\
=\int_{A^2} \bkt{u,y+z}\tilde \phi(y)\tilde \phi (z)\tilde \phi(x-z)\bkt{z}\operatorname{d}\! y\operatorname{d}\! z
=\int_{A} \bkt{u,z}\phi(u)\tilde \phi (z)\tilde \phi(x-z)\bkt{z}\operatorname{d}\! z.
\end{multline}
Next, we divide by $\tilde \phi(x)\phi(u)$ and take the inverse Fourier transform on variable $u$:
\begin{equation}\label{raman}
\int_A\frac{\phi (u+x)}{\phi(u)}\bkt{u;-z}\operatorname{d}\! u=\frac{\tilde \phi (z)\tilde \phi(x-z)\bkt{z}}{\tilde\phi(x)},
\end{equation}
and the complex conjugate
\begin{equation}
\int_A\frac{\phi(u)}{\phi (u+x)}\bkt{u;z}\operatorname{d}\! u=\frac{\overline{\tilde \phi }(z)\overline{\tilde \phi}(x-z)}{\overline{\tilde \phi}(x)\bkt{z}}.
\end{equation}
 Multiplying the last equality by $\bkt{x;z}$ and shifting the integration variable $u\mapsto u-x$, we observe that the left hand side of the obtained equality is of the same form 
as the integral in \eqref{raman} with negated $x$ and $z$. Thus, we obtain the following equality:
\begin{equation}
\frac{\tilde \phi (-z)\tilde \phi(-x+z)\bkt{z}}{\tilde\phi(-x)}=\frac{\overline{\tilde \phi }(z)\overline{\tilde \phi}(x-z)}{\overline{\tilde \phi}(x)\bkt{z}}\bkt{x;z}.
\end{equation}
Shifting the variable $x\mapsto x+z$ and cancelling some simple factors, we rewrite the latter equality in the form
\begin{equation}
\frac{h(x+z)}{h(x)h(z)}=\bkt{x;z},\quad h(x)\equiv \frac{\overline{\tilde \phi }(x)}{\tilde \phi (-x)},
\end{equation}
with the solution
\begin{equation}
h(x)=\bkt{x} /g(x),
\end{equation}
where $g(x)$ is a group homomorphism from the additive group $A$ to the multiplicative group $\REALS^\times$. Let us denote by $\varepsilon\in A$ the element of order two corresponding to the character associated to the sign of $g(x)$, i.e.
\begin{equation}
\bkt{\varepsilon;x}\equiv \frac{g(x)}{|g(x)|}.
\end{equation}
Then we have an equality of the form
\begin{equation}\label{analyt}
\tilde \phi (x)\bkt{\varepsilon;x}|g(x)|=\overline{\tilde \phi (-x)\bkt{x}},
\end{equation}
from which, by multiplying both sides by $\bkt{y;x}$, integrating over $x$, and using the unitarity of $\phi(y)$, we remark that the following integral is at least conditionally convergent for all values of the real parameter  $t$ in the closed unit interval, $t\in[0,1]$, 
\begin{equation}
\phi_t(y)\equiv\int_A\tilde \phi (x)|g(x)|^{t}\bkt{y;x}\operatorname{d}\! x.
\end{equation}
Indeed, for $t=0$, we have
$\phi_0(y)=\phi(y)$. On the other hand, for $t=1$, from \eqref{analyt}, by using unitarity of $\phi(z)$ and the inversion relation, we calculate
\begin{multline}
 \phi_1(y+\varepsilon)=\int_A\bkt{y;x}\overline{\tilde \phi (-x)\bkt{x}}\operatorname{d}\! x
 =\int_{A^2}\bkt{y;x}\overline{\phi(z)\bkt{z;x}\bkt{x}} \operatorname{d}\! z\operatorname{d}\! x\\
  =\int_{A^2}\frac{\bkt{y;x}\bkt{-z;x}}{\phi(z)\bkt{x}}\operatorname{d}\! z\operatorname{d}\! x=
   \frac1{\phi(0)^{2}}\int_{A^2}\frac{\phi(-z)\bkt{y-z;x}}{\bkt{z}\bkt{x}}\operatorname{d}\! z\operatorname{d}\! x\\
 =
   \frac1{\phi(0)^{2}}\int_{A^2}\frac{\phi(z)\bkt{y+z;x}}{\bkt{z}\bkt{x}}\operatorname{d}\! x\operatorname{d}\! z 
    =
   \frac1{\gamma\phi(0)^{2}}\int_{A}\frac{\phi(z)\bkt{y+z}}{\bkt{z}}\operatorname{d}\! z\\ 
      =
   \frac{\bkt{y}}{\gamma\phi(0)^{2}}\int_{A}\phi(z)\bkt{y;z}\operatorname{d}\! z =
   \frac{\bkt{y}\tilde\phi(-y)}{\gamma\phi(0)^{2}},\quad \gamma\equiv\int_A\bkt{x}\operatorname{d}\! x.
\end{multline}
\end{proof}

\begin{example} In the particular case of $A=\REALS\oplus(\INTEGERS/\myN\INTEGERS)$, with $\phi=\gfad_\theta$, $\varepsilon=(0,M)$ for some $M\in\INTEGERS/\myN\INTEGERS$ of order two such that $M=0$ if $\myN\ne0\pmod 8$, 
\begin{equation}
|g(x,n)|=e^{2\pi\imun c_\theta x/\sqrt{\myN}},
\end{equation}
 the function $\phi_t(x)$ is the analytically continued to complex domain function $\gfad_\theta$: 
\begin{equation}
\phi_t(x,n)=\gfad_\theta\!\left(x+tc_\theta\myN^{-1/2},n\right)
\end{equation}
The Fourier transformation formula in this case takes the form
\begin{multline}\label{fourierf}
 \int_{\lca}\gfad_\theta(x,m)\bkt{y,n;x,m}^{-1}\operatorname{d}(x,m)\\=
\gfad_\theta(-y+c_\theta\myN^{-1/2},-n+M)\bkt{y,n}^{-1}e^{\pi\imun(\myN-4c_\theta^2\myN^{-1})/12}.
\end{multline}
\end{example}

\subsection{Construction of the partition function}
Now suppose we have a solution to the \emph{shaped} pentagon relation over a self-dual LCA group $A$, which corresponds to a function of the Faddeev type of the form
\begin{equation}
f_j(x)=\psi_{a_j,c_j}(x),\quad j\in [4],\quad x\in A,
\end{equation}
where $\psi_{a,c}(x)$ is the ``charged" quantum dilogarithm,
and where parameters $a_j$, $c_j$ are positive real numbers satisfying the conditions
\begin{equation}
b_j\equiv 1-a_j-c_j>0,\quad  j\in [4],
\end{equation}
and
\begin{equation}\label{pentagonconditions}
 a_1=a_0+a_2,\ a_3=a_2+a_4,\ c_1=c_0+a_4,\ c_3=a_0+c_4,\ c_2= c_1+c_3.
\end{equation}
These conditions are satisfied by dihedral angles of ideal hyperbolic tetrahedra. We also assume that this solution realizes the tetrahedral symmetries which boil down to two relations of the form
\begin{equation}
\overline{\psi_{a,c}(x)}= \bkt{x} \psi_{c,a}(-x)\xi,
\end{equation}
for some $\xi\in\TORUS$ independent of $x$,
and 
\begin{equation}
(\mathsf{F}\psi_{a,c})(x)\equiv\int_A\psi_{a,c}(y)\bkt{y;x}\operatorname{d}\! y=\gamma^{-1/3}\bkt{x}\psi_{c,b}(-x),\quad b\equiv \lambda-a-c,
\end{equation}
where
\begin{equation}
\gamma\equiv \int_A\bkt{x}dx\in\TORUS.
\end{equation}

A \emph{triangulation} $X$ is a $\Delta$-triangulation in the sense of A.~Hatcher~\cite{MR1867354}. We denote by $\Delta_i(X)$ the set of $i$-dimensional cells of $X$. We refer to~\cite{AndersenKashaev2011} for further notation and the detailed description of the combinatorial setting of shaped triangulated oriented pseudo (STOP) 3-manifolds. 

Let $T\subset\REALS^3$ be a shaped  tetrahedron with vertex ordering mapping
\begin{equation}
v\colon \{0,1,2,3\}\to \Delta_0(T).
\end{equation}
The  \emph{Boltzmann weight function} of $T$ is a map
$B_T\colon A^{\Delta_1(T)}\to\COMPLEXS$  defined by the formula
\begin{equation}
B_T(x)=g_{\alpha_1,\alpha_3}(x_{01}+x_{23}-x_{03}-x_{12},\, x_{03}+x_{12}-x_{02}-x_{13})
\end{equation}
if $T$ is positive and complex conjugate thereof, otherwise.
 Here 
 \begin{equation}
x_{ij}\equiv x(v_iv_j),\quad\alpha_i\equiv\alpha_T(v_0v_i),
\end{equation}
\begin{equation}\label{eq:gac}
g_{a,c}(x,y)\equiv\bkt{x;-y/2}\int_B(\mathsf{F}\psi_{a,c})(y+b)\overline{\bkt{b}}\bkt{b;-x}\operatorname{d}\!b,
\end{equation}
Likewise, the Boltzmann weight function of $X$ is a map
$B_X\colon A^{\Delta_1(X)}\to\COMPLEXS$  defined as the product of tetrahedral Boltzmann weight functions
\begin{equation}
B_X(x)=\prod_{T\in\Delta_3(X)}B_T\left(x\vert_{\Delta_1(T)}\right)
\end{equation}

\begin{theorem}\label{main}
Let $X$ be a closed STOP 3-manifold. Then 
\begin{itemize}
\item there exists a unique function
 $\slashed{B}_X\colon A/B^{\Delta_1(X)}\to \COMPLEXS$ such that 
 \begin{equation}
B_X(x)=\slashed{B}_X(\pi\circ x),
\end{equation}
 where $\pi\colon A\to A/B$ is the canonical projection;
 \item the absolute value of the \emph{partition function}
\begin{equation}\label{eq:si}
Z(X)\equiv\int_{A/B^{\Delta_1(X)}}\slashed{B}_X(x)dx
\end{equation}
is invariant under shaped $2-3$ and $3-2$ Pachner moves, the angle gauge transformations, and under all changes of edge orientations in the $\Delta$ triangulation.
\end{itemize}
\end{theorem}
\begin{remark}
The state integral~\eqref{eq:si} extends to arbitrary (non closed) STOP 3-manifolds. In that case one has to integrate over only the state variables living on the internal edges. The result is a  continuous section of a complex line bundle over $A/B^{\Delta_1(\partial X)}$.
\end{remark}
\begin{remark}
With additional normalization factors associated with regular vertices, the partition function~\eqref{eq:si} can be shown to be invariant under shaped $2-0$  ($0-2$) moves which remove (add) regular vertices. In particular, this additional property permits to define link invariants in compact oriented 3-manifolds by using arbitrary H-triangulations.
\end{remark}

\section{A quantum dilogarithm over $\REALS\oplus(\INTEGERS/\myN\INTEGERS)$ }

\label{QDN}

 For any positive integer $\myN$, we denote by $\lca$ the LCA group $\REALS\oplus(\INTEGERS/\myN\INTEGERS)$ with the normalized Haar measure $\operatorname{d}(x,n)$ corresponding to the integral
\begin{equation}
\int_{\lca}f(x,n)\operatorname{d}(x,n)\equiv\frac1{\sqrt{\myN}}\sum_{n\in\INTEGERS/\myN\INTEGERS}\int_{\REALS}f(x,n)\operatorname{d}\!x
\end{equation}
for any (integrable) complex valued function $f\colon \lca\to\COMPLEXS$.  Notice that the group $\lca $ with $\myN=1$ is naturally identified with real axis $\REALS$ with its standard Lebesgue measure.
We will work with the following Gaussian exponential on $\lca$:
\begin{equation}
\bkt{x,n}=e^{\pi\imun x^2}e^{-\pi\imun n(n+N)/N}
\end{equation}
which corresponds to the following Fourier coefficients:
\begin{equation}
\bkt{x,m;y,n}=e^{2\pi\imun xy}e^{-2\pi\imun mn}.
\end{equation}
 Let us fix $\la\in\TORUS$ lying in the interior of the first quadrant, i.e. 
\begin{equation}
|\la|=1,\quad \Re\la>0,\quad\Im\la\ge0,
\end{equation} 
and denote
\begin{equation}
\cla\equiv \imun\left(\la+\la^{-1}\right)/2.
\end{equation}
We rewrite the function defined in \eqref{dfun} and \eqref{mfun} as the following explicit formula
\begin{equation}\label{gfad}
\gfad_{\theta} (x,n)=\prod_{j=0}^{\myN -1}\fad_{\theta}\!\left(\frac{x}{\sqrt{\myN }}+(1-\myN ^{-1})\cla-\imun{\la}^{-1}\frac{j}{\myN }-\imun\la\left\{\frac{j+n}{\myN }\right\}\right)
\end{equation}
where $\{x\}$ denotes the fractional part of $x$ and $\fad_{\theta}(z)$ is Faddeev's quantum dilogarithm over $\REALS$ defined for $\Im\theta>0$ 
as the following ratio of infinite products
\begin{equation}
\fad_{\theta}(z)\equiv\left(e^{2\pi\la(z+\cla)};e^{2\pi\imun\la^2}\right)_\infty/\left(e^{2\pi\la^{-1}(z-\cla)};e^{-2\pi\imun\la^{-2}}\right)_\infty,
\end{equation}
see also \eqref{mfun},  \eqref{ffun}. The derivation of \eqref{gfad} from \eqref{dfun} and \eqref{mfun} is based on the following obvious identity for infinite products:
\begin{equation}
(x;q)_\infty=\prod_{j=0}^{\myN-1}(xq^j;q^\myN)_\infty.
\end{equation}

\begin{remark}
 In formula~\eqref{gfad}, one can replace the arguments by their complex conjugates without changing the function itself. This can be shown by the following change of the summation index:
 \begin{equation}
j\mapsto j'=\myN-1-\left\{\frac{j+n}{\myN}\right\}.
\end{equation}

\end{remark}
\subsection{The inversion relation}
\begin{lemma}
The following inversion relation holds true:
 \begin{equation}\label{inversion}
\gfad_{\theta}(x,n)\gfad_{\theta}(-x,-n)=\bkt{x,n}e^{-\pi\imun(\myN +2\cla^2 \myN ^{-1})/6}.
\end{equation}
\end{lemma}
\begin{proof}
Denoting
\begin{equation}
a_j(x,n)\equiv\frac{x}{\sqrt{\myN }}+(1-\myN ^{-1})\cla-\imun{\la}^{-1}\frac{j}{\myN }-\imun\la\left\{\frac{j+n}{\myN }\right\},
\end{equation}
we remark that
\begin{multline}
 a_j(-x,-n)+a_{\myN-1-j}(x,n)\\=2(1-\myN ^{-1})\cla-\imun{\la}^{-1}(1-\myN^{-1})-\imun\la\left(\left\{\frac{j-n}{\myN }\right\}
 +\left\{\frac{\myN-1-j+n}{\myN }\right\}\right)=0
\end{multline}
as the result of the identity
\begin{equation}
\left\{\frac{j}{\myN }\right\}
 +\left\{-\frac{j+1}{\myN }\right\}=1-N^{-1},\quad\forall j\in\INTEGERS.
\end{equation}
Thus,
\begin{multline}
 \gfad_{\theta}(x,n)\gfad_{\theta}(-x,-n)=\prod_{j=0}^{\myN-1}\fad_{\theta}(a_j(x,n))\fad_{\theta}(a_{\myN-1-j}(-x,-n))\\=
\prod_{j=0}^{\myN-1}\fad_{\theta}(a_j(x,n))\fad_{\theta}(-a_{j}(x,n))=\prod_{j=0}^{\myN-1}\fad_{\theta}(0)^2e^{\pi\imun a_j(x,n)^2}\\
=\fad_{\theta}(0)^{2\myN}e^{\pi\imun\sum_{j=0}^{\myN-1} a_j(x,n)^2},
\end{multline}
where we used the inversion relation for Faddeev's function
\begin{equation}
\fad_{\theta}(z)\fad_{\theta}(-z)=e^{\pi\imun z^2}\fad_{\theta}(0)^2.
\end{equation}
Formula~\eqref{inversion} follows now by direct calculation and the fact that
\begin{equation}
\fad_{\theta}(0)=e^{-\pi\imun(1+2\cla^2)/12}.
\end{equation}
\end{proof}
\subsection{The five term operator relation}
In the Hilbert space $L^2(\REALS)$ we consider a pair of  Heisenberg self-adjoint operators $\mathsf{p}$ and $\mathsf{q}$ satisfying the commutation relation
\begin{equation}\label{heisenberg-pq}
[\mathsf{p},\mathsf{q}]\equiv\mathsf{p}\mathsf{q}-\mathsf{q}\mathsf{p}=\frac{1}{2\pi\imun}
\end{equation}
so that  their spectra  are given by the set $\REALS\subset \COMPLEXS$. 
Similarly,
in the Hilbert space $L^2(\INTEGERS/\myN\INTEGERS)\simeq\COMPLEXS^\myN$, we consider a pair of unitary Weyl operators $\mathsf{X}$ and $\mathsf{Y}$ satisfying the relations
\begin{equation}\label{weyl-xy}
\mathsf{Y}\mathsf{X}=e^{2\pi\imun/\myN} \mathsf{X}\mathsf{Y}, \quad \mathsf{X}^\myN=\mathsf{Y}^\myN=1.
\end{equation}
The spectra of both $\mathsf{X}$ and $\mathsf{Y}$ are given by one and the same set $\TORUS_\myN\subset\TORUS\subset\COMPLEXS$ of all $\myN$-th complex roots of unity. 
By using the natural group isomorphism
\begin{equation}
\operatorname{L}_{\myN}\colon \TORUS_\myN\to \INTEGERS/\myN\INTEGERS,
\end{equation}
whose inverse is given by the primitive character
\begin{equation}
\operatorname{L}_{\myN}^{-1}(n)=e^{2\pi\imun n/N},
\end{equation}
and the spectral theorem, we can define the operator function
$\operatorname{L}_{\myN}(\mathsf{A})$
with ``values" in $\INTEGERS/\myN\INTEGERS$ for any operator $\mathsf{A}$ of order $\myN$ in an arbitrary Hilbert space. It formally satisfies the relation
\begin{equation}
\mathsf{A}=e^{2\pi\imun \operatorname{L}_{\myN}(\mathsf{A})/N},
\end{equation}
and by using relations~\eqref{weyl-xy}, we can even write a Heisenberg commutation relation
\begin{equation}\label{heisenberg-mn}
[\operatorname{L}_{\myN}(\mathsf{X}),\operatorname{L}_{\myN}(\mathsf{Y})]\equiv\operatorname{L}_{\myN}(\mathsf{X})\operatorname{L}_{\myN}(\mathsf{Y})-\operatorname{L}_{\myN}(\mathsf{Y})\operatorname{L}_{\myN}(\mathsf{X})=\frac{(\myN-1)\myN}{2\pi\imun}
\end{equation}
which is formal and does not make strict sense but sometimes it is convenient in calculations. For example, by applying the  Baker--Campbell--Hausdorff formula, we calculate a product of Weyl operators as follows,
\begin{multline}
\mathsf{X}\mathsf{Y}=e^{2\pi\imun \operatorname{L}_{\myN}(\mathsf{X})/N}e^{2\pi\imun \operatorname{L}_{\myN}(\mathsf{Y})/N}=e^{2\pi\imun( \operatorname{L}_{\myN}(\mathsf{X})+\operatorname{L}_{\myN}(\mathsf{Y}))/N}e^{(2\pi\imun/\myN)^2[\operatorname{L}_{\myN}(\mathsf{X}),\operatorname{L}_{\myN}(\mathsf{Y})]/2}\\
=e^{2\pi\imun( \operatorname{L}_{\myN}(\mathsf{X})+\operatorname{L}_{\myN}(\mathsf{Y}))/N}e^{\pi\imun(N-1)/\myN}=-e^{2\pi\imun( \operatorname{L}_{\myN}(\mathsf{X})+\operatorname{L}_{\myN}(\mathsf{Y}))/N}e^{-\pi\imun/\myN},
\end{multline}
and conclude that the operator 
$ -e^{\pi\imun/N}\mathsf{X}\mathsf{Y}$
is of order $\myN$, and we also obtain the formula
\begin{equation}
\operatorname{L}_{\myN}\!\left( -e^{\pi\imun/N}\mathsf{X}\mathsf{Y}\right)=\operatorname{L}_{\myN}(\mathsf{X})+\operatorname{L}_{\myN}(\mathsf{Y}).
\end{equation}
All this corresponds to addition of u-operators discussed in the Introduction in the context of arbitrary self-dual LCA groups.

The Hilbert space $L^2(\lca)$ is naturally isomorphic to the tensor product space  $L^2(\REALS)\otimes L^2(\INTEGERS/\myN\INTEGERS)$, and thus we can consider self-adjoint operators $\mathsf{p}$, $\mathsf{q}$ and unitary operators $\mathsf{X}$, $\mathsf{Y}$ which satisfy the relations~\eqref{heisenberg-pq}, \eqref{weyl-xy} as well as the cross relations
\begin{equation}\label{cross-rel}
[\mathsf{p},\mathsf{X}]=[\mathsf{p},\mathsf{Y}]=[\mathsf{q},\mathsf{X}]=[\mathsf{q},\mathsf{Y}]=0.
\end{equation}
In particular, again by using the spectral theorem, for any function $\operatorname{f}\colon \lca\to\COMPLEXS$, we can associate an operator function $\slashed{\operatorname{f}}(\mathsf{x},\mathsf{A})\equiv\operatorname{f}(\mathsf{x},\operatorname{L}_{\myN}(\mathsf{A}))$ for any commuting pair of operators $\mathsf{x}$  and $\mathsf{A}$ where the former is self-adjoint and the latter is of order $\myN$.
\begin{lemma} Let self-adjoint operators $\mathsf{p}$ and $\mathsf{q}$, and unitary operators $\mathsf{X}$ and $\mathsf{Y}$ satisfy relations~\eqref{heisenberg-pq}, \eqref{weyl-xy}, and \eqref{cross-rel}.
Then there exists an element $\mu\in\TORUS$ such that the following five term operator relation is satisfied
\begin{equation}\label{five}
\mu\slashed{\gfad}_{\theta}\!\left(\mathsf{p},\mathsf{X}\right)\slashed{\gfad}_{\theta}\!\left(\mathsf{q},\mathsf{Y}\right)=\slashed\gfad_{\theta}\!\left(\mathsf{q},\mathsf{Y}\right)\slashed\gfad_{\theta}\!\left(\mathsf{p}+\mathsf{q}, -e^{\pi\imun/N}\mathsf{X}\mathsf{Y}\right)\slashed\gfad_{\theta}\!\left(\mathsf{p},\mathsf{X}\right).
\end{equation}
\end{lemma}
\begin{proof} Let us use relation~\eqref{five} as a definition of a unitary operator $\mu$. It suffices to show that this operator is multiple of the identity operator.

Two operators
 \begin{equation}
\mathsf{U}\equiv \mathsf{Y}e^{2\pi\theta\mathsf{q}/\sqrt{\myN}},\quad \mathsf{V}\equiv \mathsf{X}e^{2\pi\theta\mathsf{p}/\sqrt{\myN}}
\end{equation}
generate a normal algebra in the sense that the Hermitian conjugate of any element of that algebra is in its commutant. Indeed, we have
\begin{multline}
\mathsf{V}\mathsf{U}^*=\mathsf{X}e^{2\pi\theta\mathsf{p}/\sqrt{\myN}}\mathsf{Y}^{-1}e^{2\pi\bar\theta\mathsf{q}/\sqrt{\myN}}\\=e^{2\pi\imun(1-|\theta|^2)/\myN}\mathsf{Y}^{-1}e^{2\pi\bar\theta\mathsf{q}/\sqrt{\myN}}\mathsf{X}e^{2\pi\theta\mathsf{p}/\sqrt{\myN}}
=\mathsf{U}^*\mathsf{V}.
\end{multline}
By using the functional difference equations for the Faddeev's function, we can show that $\mu$ commutes with $\mathsf{U}$, $\mathsf{V}$ and their Hermitian conjugates. In particular, it commutes
with $UU^*=e^{2\pi\Re(\theta)\mathsf{q}/\sqrt{\myN}}$ and $VV^*=e^{2\pi\Re(\theta)\mathsf{p}/\sqrt{\myN}}$. The latter, being real exponentials of self-adjoint operators, imply that $\mu$ commutes with both $\mathsf{p}$ and $\mathsf{q}$, i.e. it is independent of $\mathsf{p}$ and $\mathsf{q}$ and thus can depend on only of $\mathsf{X}$ and $\mathsf{Y}$. However, as $\mu$ commuts with $\mathsf{U}$ and $\mathsf{V}$ separately, we conclude that it is a scalar.
\end{proof}
\subsection{Charging}
Let $a,b,c$ be three positive real numbers satisfying the linear condition
\begin{equation}
a+b+c=\frac1{\sqrt{\myN}}.
\end{equation}
Following closely our theory of Teichm\"uller TQFT of \cite{MR3227503},
we define a \emph{charged} function
\begin{equation}
\psi_{\sqrt{\myN}a,\sqrt{\myN}c}(x,n)\equiv e^{-2\pi\imun\cla a x}/\gfad_\la\!\left(x-\cla(a+c),n\right)
\end{equation}
for which we have the following formulae:
 \begin{equation}\label{f1}
\tilde\psi_{\sqrt{\myN}a,\sqrt{\myN}c}'(x,n)=\psi_{\sqrt{\myN}c,\sqrt{\myN}b}(x,M+n)e^{-\pi\imun\cla^2a(a+2c)}e^{-\pi\imun(\myN-4\cla^2\myN^{-1})/12}
\end{equation}
which follows from the Fourier transformation formula \eqref{fourierf},
\begin{multline}\label{f2}
 \overline{\psi_{\sqrt{\myN}a,\sqrt{\myN}c}(x,n)}=\psi_{\sqrt{\myN}c,\sqrt{\myN}a}(-x,-n)\bkt{x,n}e^{\pi\imun\cla^2(a+c)^2}e^{-\pi\imun(\myN+2\cla^2\myN^{-1})/6}\\
= \tilde\psi_{\sqrt{\myN}b,\sqrt{\myN}c}(-x,-n+M)e^{2\pi\imun Mn/N}e^{-2\pi\imun\cla^2ab}e^{-\pi\imun(\myN-4\cla^2\myN^{-1})/12},
\end{multline}
which follow from the inversion formula~\eqref{inversion} and the previous formula, and
\begin{multline}\label{f3}
 \overline{\tilde\psi'_{\sqrt{\myN}a,\sqrt{\myN}c}(x,n)}\\=\psi_{\sqrt{\myN}b,\sqrt{\myN}c}(-x,-n+M)\bkt{x,n+M}e^{-2\pi\imun\cla^2ab}e^{-\pi\imun(\myN-4\cla^2\myN^{-1})/12},
\end{multline}
which a consequence of the previous two formulas. If the element $M\in\INTEGERS/\myN\INTEGERS$  is trivial (which is the case for if $\myN$ is not multiple of 8) then formulae \eqref{f1}--\eqref{f3} are direct analogues of those in the end of Section~4 of  \cite{MR3227503}, and thus the analogues of charged pentagon relation, the exponential decay properties at infinity, as well as the fundamental lemma  of \cite{MR3227503} are also valid for this new solution. All that means that we can calculate partition functions of shaped triangulated 3-manifolds along the line described in \cite{MR3227503}.

\def\cprime{$'$} \def\cprime{$'$}

%\bibliographystyle{plain}
%\bibliography{bibliocopy}
%\bibliography{/Users/rinatkashaev/Dropbox/LatexStaff/biblio}{}
\end{document}